\documentclass[twoside,11pt,abbrvbib]{article}

\usepackage{amsmath}%
\usepackage{amsfonts}%
\usepackage{amssymb}%
\usepackage{mathrsfs}
\usepackage{amsthm}
\usepackage{jmlr2e}
\usepackage[capitalize]{cleveref}
\usepackage{overpic}
\usepackage{mathtools}
\usepackage{enumitem}
\usepackage{comment} 

\newtheorem{assumption}{Assumption}

\newcommand{\R}[0]{\mathbb{R}}

\newcommand{\N}[0]{\mathbb{N}}

\newcommand{\supp}{\operatorname{supp}}
\newcommand{\dist}{\operatorname{dist}}
\newcommand{\diam}{\operatorname{diam}}

\newcommand{\Span}{\operatorname{span}}
\providecommand{\abs}[1]{\lvert#1\rvert}
\providecommand{\norm}[1]{\lVert#1\rVert}
\providecommand{\set}[1]{\{#1\}}

\renewcommand{\d}{\, \textup{d}}
\newcommand{\U}{\mathcal{U}}
\newcommand{\D}{\mathcal{D}}
\newcommand{\F}{\mathscr{F}}
\newcommand{\mtx}[1]{\mathbf{#1}}
\DeclareMathOperator{\Tr}{Tr}
\DeclareMathOperator{\HS}{HS}

\DeclareMathOperator*{\esssup}{ess\,sup}

\usepackage{xcolor}

\usepackage{lastpage}
\jmlrheading{23}{2022}{1-\pageref{LastPage}}{4/22}{8/22}{22-0433}{Nicolas Boull\'e, Seick Kim, Tianyi Shi, and Alex Townsend}

\ShortHeadings{Learning parabolic Green's functions}{Boull\'e, Kim, Shi, and Townsend}
\firstpageno{1}

\begin{document}

\title{Learning Green's functions associated with time-dependent partial differential equations}

\author{\name Nicolas Boull\'e \email boulle@maths.ox.ac.uk \\
       \addr Mathematical Institute\\
       University of Oxford\\
       Oxford, OX2 6GG, UK
       \AND
       \name Seick Kim \email kimseick@yonsei.ac.kr \\
       \addr Department of Mathematics\\
       Yonsei University\\
       Seoul, 03722, ROK
       \AND
       \name Tianyi Shi \email ts777@cornell.edu \\
       \addr Center for Applied Mathematics\\
       Cornell University\\
       Ithaca, NY 14853, USA
	   \AND
       \name Alex Townsend \email townsend@cornell.edu \\
       \addr Department of Mathematics\\
       Cornell University\\
       Ithaca, NY 14853, USA       
       }

\editor{Michael Mahoney}

\maketitle

\begin{abstract}
Neural operators are a popular technique in scientific machine learning to learn a mathematical model of the behavior of unknown physical systems from data. Neural operators are especially useful to learn solution operators associated with partial differential equations (PDEs) from pairs of forcing functions and solutions when numerical solvers are not available or the underlying physics is poorly understood. In this work, we attempt to provide theoretical foundations to understand the amount of training data needed to learn time-dependent PDEs. Given input-output pairs from a parabolic PDE in any spatial dimension $n\geq 1$, we derive the first theoretically rigorous scheme for learning the associated solution operator, which takes the form of a convolution with a Green's function $G$. Until now, rigorously learning Green's functions associated with time-dependent PDEs has been a major challenge in the field of scientific machine learning because $G$ may not be square-integrable when $n>1$, and time-dependent PDEs have transient dynamics. By combining the hierarchical low-rank structure of $G$ together with randomized numerical linear algebra, we construct an approximant to $G$ that achieves a relative error of $\smash{\mathcal{O}(\Gamma_\epsilon^{-1/2}\epsilon)}$ in the $L^1$-norm with high probability by using at most $\smash{\mathcal{O}(\epsilon^{-\frac{n+2}{2}}\log(1/\epsilon))}$ input-output training pairs, where $\Gamma_\epsilon$ is a measure of the quality of the training dataset for learning $G$, and $\epsilon>0$ is sufficiently small.
\end{abstract}


\section{Introduction}

Machine learning, numerical analysis, and scientific computing are successfully combining in the field of scientific machine learning to integrate data and prior knowledge of physical laws to solve inverse problems using deep learning~\citep{karniadakis2021physics}. The flexibility of neural network architectures and exceptional generalization errors, make neural networks ideal for scientific machine learning. On the other hand, it is challenging to mathematically justify the success of deep learning in this context.

A central topic in scientific machine learning is to discover partial differential equations (PDEs), which are mathematical models describing the relations between variables of a system and their spatial and temporal derivatives, directly from simulations or experimental data. This leads to a wide range of applications in weather forecasting and climate science~\citep{rasp2018deep,zanna2020data}, biology~\citep{alber2019integrating,raissi2020hidden}, and physics~\citep{karniadakis2021physics,kochkov2021machine,kutz2017deep,Chene2102721118}. Traditionally, PDEs are derived from mechanistic insights using conservation laws, minimum energy principles, or empirical observations~\citep{evans10}. With the rapid development of deep learning and the vast collection of experimental results from sensors, we are beginning an exciting new era of uncovering unknown PDE models directly from data. Still, learning time-dependent PDEs is challenging because of transient dynamics.

We aim to provide the first theoretical results to characterize how much training data is needed to learn a time-dependent PDE, and close a theoretical gap with recent data-driven methods~\citep{boulle2021data,gin2020deepgreen,li2020fourier,lu2021learning}. Our main result, described in \cref{sec_main_result}, exploits and draws connections with standard tools from numerical analysis, such as approximation theory and numerical linear algebra. While we exclusively focus on theory, the insights provided by this work will be of interest to a broader audience in scientific machine learning and motivate future empirical works and novel physics-informed neural network architectures.

\subsection{Parabolic Partial Differential Equations}

Throughout this paper, we consider a class of time-dependent PDEs called parabolic partial differential operators. A parabolic partial differential operator defined on a bounded spatial domain $\Omega\subset\mathbb{R}^n$ for some $n\geq 1$ with Lipschitz smooth boundary takes the form:
\begin{equation} \label{eq:parabolic}
\mathcal{P}u\coloneqq u_t -\nabla \cdot \left(A(x,t)\nabla u\right) = f(x,t),\quad x\in \Omega, \, t\in[0,T], \quad 0<T<\infty.
\end{equation}
Here, for every $x\in \Omega$ and $t\in[0,T]$, the matrix $A(x,t)\in\mathbb{R}^{n\times n}$ is symmetric positive definite with bounded coefficient functions in $L^\infty(\U)$, where $\U\coloneqq\Omega\times [0,T]$, and satisfies the uniform parabolicity condition (see~\cref{eq_uni_parab}). Here, $L^\infty(\U)$ is the space of measurable functions defined on $\U$ that have a bounded essential supremum. In this manuscript, we also consider two other $L^p$ spaces, the space of absolutely integrable functions, $L^1(\U)$, and squared-integrable functions, $L^2(\U)$. We emphasize that the regularity requirements on the parabolic PDE are very weak. The function $f$ in \cref{eq:parabolic} is called the forcing term of the PDE while $u$ is the corresponding system's response or solution. Parabolic PDEs model a wide variety of time-dependent phenomena, including heat conduction, particle diffusion, and option pricing. 

The goal of most PDE learning tasks is to learn the solution operator that maps forcing terms to responses, given training data $\{(f_j,u_j)\}_{j=1}^N$~\citep{boulle2021data,gin2020deepgreen,kovachki2021neural,li2020neural,li2020multipole,li2020fourier,lu2021learning,wang2021learning}. Associated with the parabolic operator $\mathcal{P}$ in~\cref{eq:parabolic} is a Green's function $G:\U\times \U\to \R^+$, which is a kernel for the solution operator~\citep{cho2012global}. In particular, the solution operator is an integral operator of the form~\citep{evans10}:
\[
u(x,t) = \int_0^T \int_\Omega G(x,t,y,s)f(y,s)\d y \d s, \quad (x,t)\in \U,
\]
where $u$ is the solution to~\cref{eq:parabolic} given the forcing term $f$. Our goal is to recover $G$ as accurately as possible from forcing functions $f_1,\ldots,f_N$ and their corresponding solutions $u_1,\ldots,u_N$, as well as the evaluation of the adjoint of $\mathcal{P}$. Since we are learning a classical mathematical object, we can gain a mechanistic understanding of the unknown parabolic PDE, and theoretical and practical performance guarantees.

\subsection{Challenges and Contributions} \label{sec_challenges}

In this paper, we derive a rigorous probabilistic algorithm to learn the Green's function $G$ associated with~\cref{eq:parabolic} from random input-output data $(f,u)$ and characterize the number of training pairs needed to learn $G$ to within a given tolerance $\epsilon$ with high probability. Since Green's functions associated with~\cref{eq:parabolic} may not be squared-integrable when $n>1$, we perform our analysis using the $L^1$-norm and obtain a rigorous learning rate for $G$ in that norm. We summarize the challenges that we face and our main contributions:

\paragraph{Low-rank structure of parabolic Green's functions.} 
It is known that Green's functions associated with elliptic PDEs in dimension $n\geq 3$ have a low-rank structure on well-separated domains and can be approximated by separable functions~\citep{bebendorf2003existence}. This property motivates the use of hierarchical matrices as a way to store, approximate, and compute the inverse of finite element stiffness matrices and discretized Green's functions in quasi-optimal complexity~\citep{bebendorf2003adaptive,bebendorf2008hierarchical,borm2003introduction,hackbusch1999sparse,hackbusch2000sparse,hackbusch2004hierarchical}.
The low-rank structure of the Green's function is heavily used in numerical solvers for elliptic PDEs, preconditioners for iterative solvers, computing Schur complements~\citep{bebendorf2003existence}, and rigorously learning Green's functions from input-output pairs~\citep{boulle2021learning}. While related works~\citep{greengard2000spectral,greengard1990fast,jiang2015efficient,li2007numerical,li2009high} exploited the compressibility of the heat kernel to build fast and accurate numerical methods for the evaluation of heat potentials, there is a lack of theoretical results that are analogous to those found in~\cite{bebendorf2003existence}. In particular, our first contribution is to prove that Green's functions associated with parabolic PDEs admit a low-rank structure on well-separated domains for any spatial dimension, extending~\citep{bebendorf2003existence}. Our analysis is based on the existence of Poincar\'e and Cacciopolli-type inequalities satisfied by the solutions of parabolic PDEs. We find that for the most efficient hierarchical partitioning of a domain, the time variable must be treated differently from spatial variables, leading to a careful hierarchical partition of the spatio-temporal domain. This result enables the approximation of the entire solution operator associated with a parabolic PDE by a hierarchical matrix, leading to efficient numerical solvers. Additionally, inspired by the hierarchical structures in elliptic PDEs, several neural network (NN) architectures using a wavelet transform are proposed to learn the solution operators of PDEs across different scales~\citep{feliu2020meta,gupta2021multiwaveletbased}, and our analysis suggests this is potentially a good idea for parabolic PDEs too.

\paragraph{Analysis in the $\mathbf{L^1}$-norm.} 
Green's functions associated with parabolic PDEs may not be squared-integrable when $n>1$, presenting an additional challenge compared to elliptic PDEs with $1\leq n\leq 3$. For example, consider the forced heat equation in dimension $n>1$ with zero homogeneous Dirichlet boundary conditions and zero initial conditions, i.e.,
\[\frac{\partial u}{\partial t} - \nabla^2 u = f(x,t), \quad u(x,0) = 0,\quad u(0,t) = 0, \quad (x,t)\in \R^n\times \R,\]
The associated Green's function is given by~\cite[Sec.~2.3.1]{evans10}
\begin{equation} \label{eq_heat_green}
G(x,t,y,s)=\frac{\Theta(t-s)}{(4\pi(t-s))^{n/2}}\exp\left(-\frac{1}{4}\frac{|x-y|^2}{t-s}\right),\quad (x,t)\neq (y,s)\in \R^n\times \R,
\end{equation}
where $\Theta(\cdot)$ is the Heaviside step function that takes the value of one for positive inputs and zero otherwise, and $|\cdot|$ is the Euclidean norm on $\Omega$. Due to the type of the time singularity along the diagonal as $s$ approaches $t$, $G$ is not a squared-integrable function. However, $G$ does have a bounded $L^1$-norm. This is a very significant theoretical challenge for rigorously learning the corresponding solution operator as $L^1$ is not a Hilbert space, contrary to $L^2$. Almost all the techniques employed in the elliptic case~\citep{boulle2021learning} exploit analogues of matrix results to Hilbert--Schmidt (HS) operators in infinite dimensions, such as the Eckart--Young--Mirsky theorem for best low-rank approximation in the Frobenius norm~\citep{eckart1936approximation}, and do not generalize to the $L^1$-norm. The best low-rank approximation problem for matrices in the entrywise $\ell_1$-norm is significantly more complicated than that in the Frobenius norm and is, in general, NP-hard~\citep{gillis2018complexity,song2017low}. We address this issue by approximating well-separated blocks of the Green's function in the $L^2$-norm, and then express the final approximation error in the $L^1$-norm using Moser's local maximum estimate~\citep{lieberman1996second}. This theory may motivate the use of NN architectures that allow for representing maps with singularities that are not square-integrable in deep learning, such as rational NNs~\citep{boulle2020rational}.

\paragraph{Quality of the training data.}
Several deep learning techniques for learning solution operators associated with PDEs assume random training and testing forcing functions $f$ in~\cref{eq:parabolic}, which are drawn from a Gaussian process (GP) $\mathcal{GP}(0,K)$ with a carefully designed covariance kernel $K$~\citep{boulle2021data,gin2020deepgreen,kovachki2021neural,li2020neural,li2020fourier,lu2021learning,wang2021learning}. For example, the GreenLearning, DeepGreen, and DeepONet methods use a squared-exponential kernel, i.e., $K(x,x')=\exp(-|x-x'|^2/(2\ell^2))$, where the length-scale parameter $\ell$ determines the smoothness of the forcing terms~\citep{boulle2021data,gin2020deepgreen,lu2021learning}. In contrast, the Neural Operator approach employs Green's functions associated with Helmholtz equations as covariance kernel, where the Helmholtz frequency varies depending on the problem considered~\citep{li2020neural,li2020fourier}. We emphasize that the choice of the covariance kernel is important in PDE learning applications and can be used to enforce prior knowledge about the PDE to obtain higher accuracy. Our main theoretical result contains a term that characterizes the quality of the random training data, i.e., the covariance kernel of the GP from which forcing terms are sampled (see~\cref{sec_main_result}). This is a step towards understanding the success of state-of-the-art PDE learning techniques and better determine the optimal covariance kernel to minimize the size of the training dataset.

Finally, we regard our work here as giving important theoretical insights and we are not proposing that our rigorous learning algorithm should replace deep learning techniques in practice. Instead, we hope this paper can benefit the state-of-the-art PDE learning techniques by suggesting different optimization algorithms based on an $L^1$ loss, improving the quality of training datasets, and designing ``physics-informed'' NN architectures that represent the singularity and low-rank structure present in Green's functions.

\subsection{Main Theoretical Results} \label{sec_main_result}

Our first result (see~\cref{thm01}) shows that Green's functions associated with parabolic operators of the form of \cref{eq:parabolic} satisfy similar separable approximation properties to Green's functions of elliptic operators~\citep[Thm.~2.8]{bebendorf2003existence} on admissible spatio-temporal domains $Q_X\times Q_Y\subset \U\times \U$. The notion of admissibility (see~\cref{sec_admissible}) ensures that the approximation results only apply to domains $Q_X\times Q_Y$ that do not contain the singular points of the Green's function. For any $0<\epsilon<1$ sufficiently small and $k \le k_\epsilon=\mathcal{O}(\lceil{\log\frac{1}{\epsilon}}\rceil^{n+3})$, we show that there exists a (low-rank) separable approximation of the form
\[G_k(x,t,y,s)=\sum_{i=1}^k u_i(x,t)  v_i(y,s),\quad (x,t)\in Q_X,\,(y,s)\in Q_Y,\]
such that
\[\norm{G-G_k}_{L^2(Q_X\times Q_Y)} \le \epsilon \norm{G}_{L^2(\hat{Q}_X\times Q_Y)}.\]
Here, $\hat{Q}_X\subset \U$ denotes a domain slightly larger than $Q_X$.

Throughout this paper, we make the following assumption that allows us to evaluate the adjoint of the parabolic operator to construct an approximant to the Green's function. In practical applications, it may not be possible to evaluate the adjoint, as backward parabolic equations are usually not well-posed~\citep{john1955numerical,miranker1961well}. However, numerical experiments suggest that deep NNs can approximate the solution operators associated with non-symmetric problems when the training data contains sufficient prior knowledge of the operator~\citep{boulle2021data,li2020fourier,lu2021learning}.

\begin{assumption} \label{ass_adjoint}
We assume that we can evaluate the action of the adjoint $\mathcal{P}^*$ of the parabolic operator $\mathcal{P}$, defined as
\[\mathcal{P}^* u = -u_t-\nabla\cdot(A(x,t)^\top\nabla u),\]
where $A^\top$ is the transpose of the coefficient matrix $A$.
\end{assumption}

Our main theoretical result, stated later in \cref{thm_learning_rate}, constructs a rigorous probabilistic scheme for learning Green's functions of parabolic operators in spatial dimension $n\geq 1$ within a relative error of $\mathcal{O}(\Gamma_\epsilon^{-1/2}\epsilon)$, with high probability, using at most $\mathcal{O}(\epsilon^{-\frac{n+2}{2}}\log(1/\epsilon))$ input-output training pairs, where $\epsilon>0$ is a sufficiently small parameter. The factor $\Gamma_\epsilon$ is defined later in \cref{sec:green_entire} and quantifies the quality of the forcing terms for learning $G$. This result provides an upper bound for the intrinsic learning rate of parabolic operators. Our theoretical construction relies on the separable approximation result for Green's functions associated with parabolic PDEs described earlier, a careful hierarchical partition of the spatio-temporal domain into well-separated blocks, and the randomized singular value decomposition (SVD) for HS operators~\citep{boulle2021learning,boulle2022generalization}.

\subsection{Related Works}

The approaches that dominate the PDE learning literature consist of discovering coefficients of the PDE~\citep{Brunton,Rudy,udrescu2020ai,udrescu2020ai2,zhang2018robust}, building reduced-order models to significantly speed up standard numerical solvers~\citep{qian2020lift,qian2021reduced}, and directly approximating the PDE solution operator by an artificial NN~\citep{boulle2021data,gin2020deepgreen,kovachki2021neural,li2020neural,li2020multipole,li2020fourier,lu2021learning,wang2021learning}. Several black-box deep learning techniques are proposed to approximate the solution operator, which maps forcing terms $f$ to observations of the associated system's response $u$ such that $\mathcal{P}(u) = f$, where $\mathcal{P}$ is the partial differential operator. These methods mainly differ in their choice of the NN architecture used to approximate the solution map. For example, Fourier Neural Operator~\citep{li2020fourier} uses a Fourier transform at each layer, DeepONet~\citep{lu2021learning} contains a concatenation of `trunk' and `branch' networks to enforce additional structure, and GreenLearning~\citep{boulle2021data} relies on rational NNs~\citep{boulle2020rational} to learn Green's functions.

On the theoretical side, most of the research has focused on the approximation theory of infinite-dimensional operators by NNs, such as the generalization of the universal approximation theorem~\citep{cybenko1989approximation} to shallow and deep NNs~\citep{chen1995universal,lu2021learning} as well as error estimates for Fourier Neural Operators and DeepONets with respect to the network width and depth~\citep{kovachki2021neural,kovachki2021universal,lanthaler2021error}. Other approaches aim to approximate the matrix of the discretized Green's functions associated with elliptic PDEs from matrix-vector multiplications by exploiting sparsity patterns or hierarchical structure of the matrix~\citep{lin2011fast,schafer2021sparse}. In addition,~\cite{de2021convergence} derived convergence rates for learning linear self-adjoint operators based on the assumption that the target operator is diagonal in the basis of the Gaussian prior. More recently and closely related to this work,~\cite{boulle2021learning} derived an intrinsic ``learning rate'' for elliptic PDEs using ideas from randomized linear algebra and low-rank approximation theory to characterize the number of training data needed to approximate the associated solution operator or Green's functions.

However, fundamental challenges of the field concern the interpretability of the discovered model to uncover new physical understanding, and performance guarantees with theoretical results. These are challenging, especially when the underlying mathematical model is time-dependent and has short-lived transient dynamics.

\subsection{Organization of the Paper}

The paper is organized as follows. We first introduce some definitions and our notation in \cref{sec_background}. Then, we prove a low-rank approximation property of Green's functions associated with parabolic operators in \cref{sec_low_rank} on separable domains. We exploit this low rank structure to bound the number of input-output pairs needed to learn Green's functions in \cref{sec_learning_rate} using the randomized SVD combined with a hierarchical partition of the temporal domain. We conclude in \cref{sec_conclusions} with a discussion of the results and future challenges.

\section{Background and the Randomized Singular Value Decomposition} \label{sec_background}

This section introduces our notation and background on low-rank functions, admissible domains, and the randomized SVD for HS operators.

\subsection{Definitions and Notation} \label{sec:notations}

Throughout this paper, $\Omega\subset\R^n$ denotes a bounded domain in spatial dimension $n\geq 1$ satisfying the \emph{uniform interior cone condition}~\citep[Chapt.~7.7]{gilbarg2001elliptic}, which is defined as follows. 
\begin{definition}[Uniform interior cone condition] \label{def_uniform_cone}
We say that $\Omega$ satisfies an interior cone condition  if there exists an angle $\theta\in(0,\pi/2)$ and a radius $r>0$ such that for every $x \in \Omega$ there exists a unit vector $\xi_x$ such that the cone 
\[
C(x,\xi_x, \theta,r) = \left\{x+\lambda y : y\in\mathbb{R}^n,\, \|y\|_2=1,\,y\cdot\xi_x \geq \cos(\theta),\, \lambda\in[0,r]\right\}
\] 
is contained in $\Omega$. Here, `$\cdot$' denotes the standard dot product in $\mathbb{R}^n$. 
\end{definition}
\noindent We note that every bounded domain with a Lipschitz smooth boundary satisfies an interior cone condition.

We consider parabolic PDEs of the form~\cref{eq:parabolic} on the domain $\U=\Omega\times [0,T]$, where $T>0$. We also assume that the symmetric coefficient matrix $A(x,t)\in \R^{n\times n}$ satisfies the uniform parabolicity condition, i.e.,~there exist two positive constants $\lambda,\Lambda>0$ such that
\begin{equation} \label{eq_uni_parab}
\lambda \abs{\xi}^2 \le A(x,t) \xi \cdot \xi \le \Lambda \abs{\xi}^2, \quad \xi\in \R^n,
\end{equation}
where $\abs{\cdot}$ denotes the discrete $\ell_2$-norm. This means that the matrix $A$ is uniformly positive definite with eigenvalues in the interval $[\lambda,\Lambda]$ By the Cauchy--Schwarz inequality, we have the following inequality: 
\[\abs{A(x,t) \xi \cdot \psi} \le \Lambda \abs{\xi} \abs{\psi}, \quad \xi,\psi\in \R^n.\]
Under these conditions, we can find a nonnegative Green's function $G(x,t,y,s)$ defined on the domain $\{(x,t,y,s)\in\U\times\U,\,(x,t)\neq (y,s)\}$ by the following relationship~\citep{cho2012global}:
\begin{subequations}
\begin{alignat*}{3}
\mathcal{P}G(x,t,y,s) &= \delta(y-x)\delta(s-t), \qquad &&(x,t)\in \U, \\
G(\cdot, t, y, s) &= 0, &&\text{on } \partial\Omega, \, t\in (0,T),\\
G(\cdot, 0, y,s) &= 0, &&\text{in } \Omega,
\end{alignat*}
\end{subequations}
where $\mathcal{P}$ is the parabolic operator defined in \cref{eq:parabolic} acting on functions in the variables $(x,t)$ and $\delta(\cdot)$ is the Kronecker delta function.

In this manuscript, we usually work on domains of the form $Q=(\Omega \cap D) \times I$, where $D$ is a bounded convex domain in $\R^n$ and $I$ is an open bounded interval of $\R^+$, and consider the following function spaces: 
\begin{enumerate}
\item The Banach space $L^1(Q)$, with norm $\|u\|_{L^1(Q)}=\int_Q |u|\d x\d t$.
\item The Hilbert space $L^{2}(Q)$, with inner product $\langle u,v\rangle_{L^2(Q)}=\int_{Q} uv\d x\d t$.
\item The Hilbert space $W^{1,0}_2(Q)$, with inner product $\langle u,v\rangle_{W^{1,0}_2(Q)}=\int_{Q} (uv+\nabla u \cdot \nabla v)\d x \d t$, consisting of all functions $u\in L^2(\U)$ with squared-integrable weak derivatives.
\item The Banach space $V_2(Q)$, defined as
\[V_2(Q) \coloneqq \left\{u \in W^{1,0}_2(Q),\,  \norm{u}_{V_2(Q)} = \esssup_{t\in I}\, \norm{u(\cdot, t)}_{L^2(\Omega \cap D)} + \norm{\nabla u}_{L^2(Q)} < \infty\right \}.\]
\end{enumerate}
The approximation error between the learned and exact Green's functions are expressed in the $L^1(\U\times \U)$-norm as Green's functions associated with parabolic PDEs are usually not squared-integrable when $n>1$ (see~\cref{sec_challenges}).

\subsection{Admissible Domains and Low-rank Functions} \label{sec_admissible}
We learn Green's functions on subdomains of $\U\times \U$ satisfying an admissibility condition so that the subdomains do not contain the singular points of the Green's functions located along the diagonal $(x,t) = (y,s)$. While the definition of strong admissible (or well-separated) domains is standard for Green's functions associated with elliptic PDEs~\citep{ballani2016matrices,bebendorf2003existence,bebendorf2008hierarchical,hackbusch2015hierarchical}, we need to adapt the definition for parabolic PDEs slightly.
Let $\beta>0$ be a constant, and consider the following metrics on $\U\times \U$:
\begin{equation} \label{eq_def_beta}
m(x,t,y,s)=\max \left(\norm{x-y}_\infty,\sqrt{\abs{t-s}/\beta}\right),\quad (x,t)\in\U,\, (y,s)\in \U,
\end{equation}
where the spatial and temporal variables are treated differently. The choice of the metric $m$ is related to the exponential term appearing in the Green's function of the heat equation (cf.~\cref{eq_heat_green}) since Green's functions associated with parabolic PDEs satisfy similar Gaussian bounds~\citep{cho2012global}.  Let $Q_X,Q_Y\subset \U$ be two non-empty domains, we can define the \emph{diameter} of $Q_X$ and the \emph{distance} between $Q_X$ and $Q_Y$ using the metric $m$ as
\begin{equation} \label{eq_diam_dom}
\diam Q_X = \!\! \sup_{(x,t),(y,s)\in Q_X} \!\! m(x,t,y,s),\qquad \dist(Q_X,Q_Y) = \!\! \inf_{(x,t)\in Q_X,(y,s)\in Q_Y} \!\!m(x,t,y,s).
\end{equation}
We use these quantities to define a partition of $\U\times\U$ so that the spatial and temporal domains scale similarly as we approach the singularity of the Green's function (see~\cref{sec_hier}).  Finally, one can combine the notions of diameter and distance for spatio-temporal domains to introduce an admissibility condition, similar to the elliptic case.
 
\begin{definition}[Admissible domains] \label{def_adm_domains}
For a fixed parameter $\rho>0$, we say that two non-empty domains $Q_X,Q_Y\subset \U$ are admissible if
\begin{equation} \label{eq_adm_domain}
\dist(Q_X,Q_Y)\geq \rho \max\{\diam Q_X,\diam Q_Y\}.
\end{equation}
Otherwise, we say that they are non-admissible.
\end{definition}

\cref{fig_admissible_domains} illustrates admissible and non-admissible subdomains of $\U=[0,1]\times[0,1]$. In particular, the spatial component of $\U$ is partitioned into four subdomains while the temporal component is partitioned into two. This ensures that all the subdomains have a similar diameter according to \cref{eq_diam_dom}. We use a similar strategy when constructing a partition of $\U\times\U$ into admissible and non-admissible subdomains (see~\cref{sec_hier}).

\begin{figure}[htbp]
\vspace{0.2cm}
\centering
\begin{overpic}[width=0.9\linewidth]{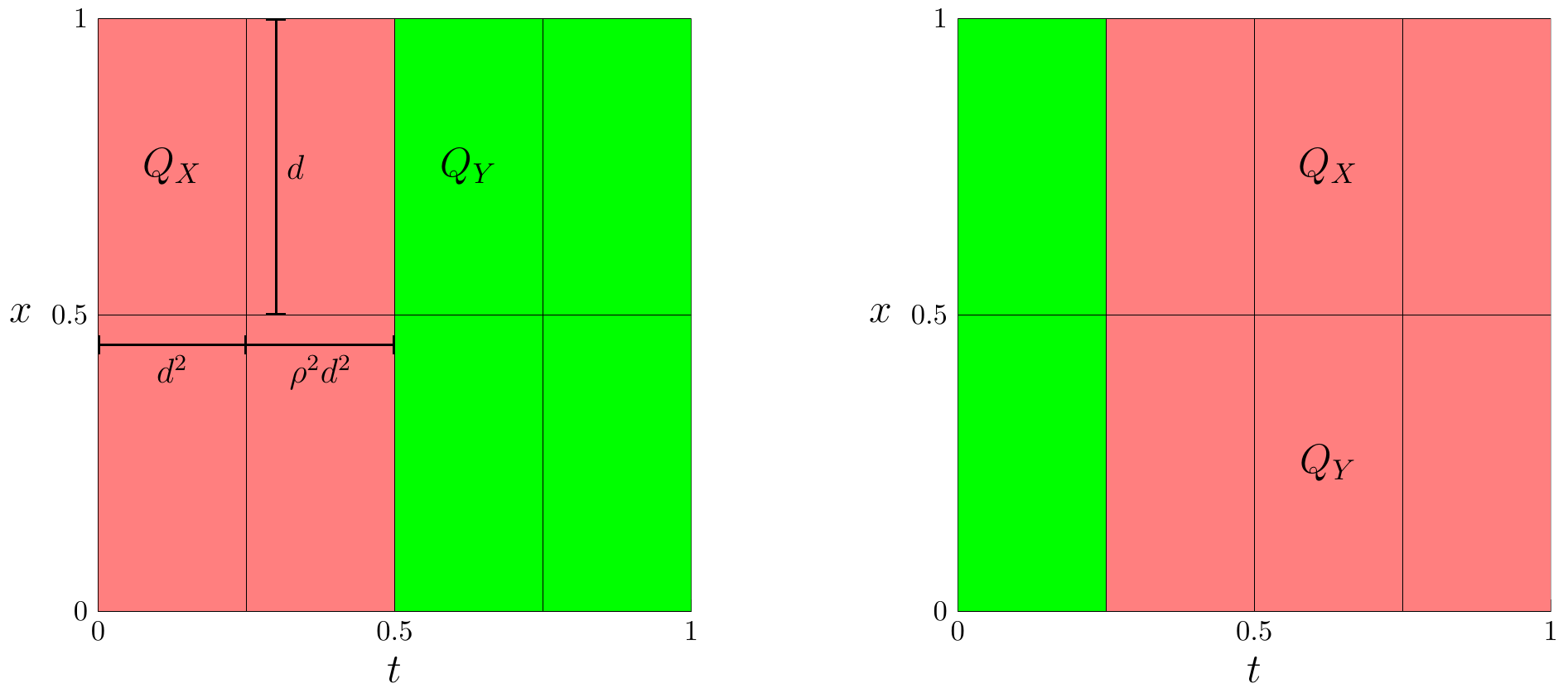}
\put(-1,42){(a)}
\put(54,42){(b)}
\end{overpic}
\caption{Admissible and non-admissible subdomains of the spatial-temporal domain $\U=[0,1]\times[0,1]$. The spatial domain is partitioned into two while the temporal domain is partitioned into four. In particular, a subdomain $Q_X=D_X\times I_X$ has diameter $d\coloneqq\diam(Q_X) = \diam(D_X)=\diam(I_X)^{1/2}$. Panels (a) and (b) highlights in green (resp.~red) the admissible domains (resp.~non-admissible) with $Q_X$ with diameter $d=1/2$ and $\rho=\beta=1$ (see~\cref{eq_def_beta,def_adm_domains}). Specifically, the subdomain $Q_Y$ is admissible with $Q_X$ in (a) and non-admissible in (b).}
\label{fig_admissible_domains}
\end{figure}

On admissible domains $Q_X\times Q_Y\subset\U\times\U$, we aim to construct approximants to the Green's function $G$ associated with \cref{eq:parabolic}. For a given accuracy $0<\epsilon<1$, we say that $G$ is of rank $k$ if there exists an integer $k=k(\epsilon)$ and a separable approximation of the form
\[G_k(x,t,y,s) = \sum_{i=1}^k u_i(x,t)v_i(y,s),\quad (x,t,y,s)\in Q_X\times Q_Y,\]
such that $\|G-G_k\|_{L^2(Q_X\times Q_Y)}\leq \epsilon \|G\|_{L^2(\hat{Q}_X\times \hat{Q}_Y)}$, where $\hat{Q}_X$ (resp.~$\hat{Q}_Y$) denotes a domain slightly larger than $Q_X$ (resp.~$Q_Y$); see \cref{thm01} for a precise definition. When $k=\mathcal{O}(\log^{\delta}(1/\epsilon))$ for some small $\delta\in \N$ as $\epsilon\rightarrow 0$, then we say that $G$ has rapidly decaying singular values on $Q_X\times Q_Y$.

\subsection{Randomized Singular Value Decomposition for Hilbert--Schmidt Operators} \label{sec:rand_svd}

Let $D_1,D_2\subset\R^n$ be two domains, a linear operator $\F:L^2(D_1)\to L^2(D_2)$ is called an HS operator if it has finite HS norm, i.e.,
\[\|\F\|_{\HS}\coloneqq \left(\sum_{j=1}^\infty\|\F e_j\|_{L^2(D_2)}^2\right)^{1/2}<\infty,\]
where $\{e_j\}_{j=1}^\infty$ is an orthonormal basis of $L^2(D_1)$. HS operators generalize the notion of matrices acting on vectors to infinite dimensions with operators acting on squared-integrable functions. Moreover, the HS norm is the continuous analogue of the Frobenius norm for matrices and $\|\F\|_{\HS}=(\sum_{j=1}^\infty\sigma_j^2)^{1/2}$, where $\sigma_1\geq \sigma_2\geq \cdots\geq 0$ denote the singular values of $\F$.

The randomized SVD is one of the most popular algorithms for constructing low-rank approximations of large matrices. Given a matrix $A$, it uses matrix-vector products with random vectors drawn from a standard Gaussian distribution to find an approximate orthonormal basis $Q$ for the column space of $A$ before computing a low-rank approximation as $QQ^\top A$~\citep{halko2011finding,martinsson2020randomized}. 

A recent generalization of the randomized SVD with random vectors drawn from a general multivariate Gaussian distribution allows its application to learn HS operators using random functions drawn from a Gaussian process~\citep{boulle2021learning,boulle2022generalization}. The randomized SVD for HS operators uses random functions drawn from a Gaussian process $\mathcal{GP}(0,K)$ with mean $(0,\ldots,0)$ and continuous symmetric positive definite covariance kernel $K:D_1\times D_1\to \R$ to construct a low-rank approximant. The kernel $K$ has positive eigenvalues $\lambda_1\geq \lambda_2\geq\cdots>0$, and there exists an orthonormal basis of $L^2(D_1)$ such that~\citep[Thm.~4.6.5]{hsing2015theoretical}
\[K(x,y) = \sum_{j=1}^\infty\lambda_j \psi_j(x)\psi_j(y), \quad x,y\in D_1.\]
The trace of the covariance kernel is defined as $\Tr(K)\coloneqq \sum_{j=1}^\infty\lambda_j<\infty$ and is finite as $K$ is continuous on $D_1\times D_1$. A random function $\omega\sim \mathcal{GP}(0,K)$ can be sampled by setting
$\smash{\omega\sim \sum_{j=1}^\infty\sqrt{\lambda_j}c_j\psi_j}$, where $c_j\sim\mathcal{N}(0,1)$ are independent and identically distributed.

Finally, it is convenient to introduce quasimatrices, which extends the notion of tall-skinny matrices to infinite dimensions~\citep{townsend2015continuous}, to formulate the randomized SVD for HS operators. Let $k\geq 1$ be an integer, a $D_1\times k$ quasimatrix $\mtx{\Omega}=\begin{bmatrix}
\omega_1 & \cdots & \omega_k
\end{bmatrix}$ is a matrix with $k$ columns, where each column $1\leq j\leq k$ is a squared-integrable function $w_j\in L^2(D_1)$. The standard matrix-vector operations generalize naturally to the applications of HS operators to quasimatrices. Then, $\F\mtx{\Omega}$ denotes the quasimatrix obtained by applying $\F$ to each column of $\mtx{\Omega}$~\citep[Sec.~2.1]{boulle2021learning}.

We can now state the results of approximating an HS operator with randomized SVD. Let $k\geq 1$ be a target rank, $p\geq 4$ an oversampling parameter, and $\mtx{\Omega}$ be a $D_1\times (k+p)$ quasimatrix, whose columns are drawn from $\mathcal{GP}(0,K)$. If $\mtx{Y}=\F\mtx{\Omega}$ and $\mtx{P}_\mtx{Y}$ is the orthogonal projection onto the vector space spanned by the columns of $\mtx{Y}$, then for $s,t\geq 1$, we have~\citep[Thm.~1]{boulle2021learning},
\begin{equation}
\label{eq:MainProbabilityBound}
\|\F-\mtx{P}_{\mtx{Y}}\F\|_{\HS}\leq \sqrt{1+ t^2s^2 \frac{3}{\gamma_k}\frac{k(k+p)}{p+1}\frac{\Tr(K)}{\lambda_1}}\,\left(\sum_{j=k+1}^\infty\sigma_j^2\right)^{1/2},
\end{equation}
with probability of failure bounded by $t^{-p} + [s e^{-(s^2-1)/2}]^{k+p}$. Here, $\gamma_k = k/(\lambda_1\Tr(\mtx{C}^{-1}))$ with $\mtx{C}_{ij}=\int_{D_1\times D_1}v_i(x)K(x,y)v_j(y)\d x\d y$ for $1\leq i,j\leq k$, where $v_j$ is the $j$th right singular vector of $\F$. The factor $0<\gamma_k\leq 1$ in \cref{eq:MainProbabilityBound} characterizes the quality of the covariance kernel to learn the HS operator $\F$. A refined bound shows that one can enforce prior information on the operator in the covariance kernel to obtain near-best approximation error~\citep[Thm.~2]{boulle2022generalization}. 

In the remainder of this manuscript, we apply the randomized SVD for HS operators to learn Green's functions associated with parabolic PDEs on admissible domains $Q_X\times Q_Y$. Green's functions restricted to admissible domains are an example of HS integral operators as the solution operator $\F$ associated with parabolic PDEs can be written as
\[
(\F f)(x,t)=\int_{Q_Y}G(x,t,y,s)f(y,s) \d y \d s,\quad f\in L^2(Q_X),\, (x,t)\in Q_X.
\]
Moreover, we can use the relation $\|\F\|_{\HS}=\|G\|_{L^2(Q_X\times Q_Y)}$ to estimate the approximation error between the Green's function and its approximant on $Q_X\times Q_Y$.

\section{Low-Rank Approximation of Parabolic Green's Functions} \label{sec_low_rank}

\cite{bebendorf2003existence} show that Green's functions associated with elliptic equations in three dimensions admit a low rank separable approximation on admissible domains. In this section, we extend this result to Green's functions associated with parabolic PDEs so that we obtain low-rank approximants on well-separated domains (see~\cref{thm01}). In particular, approximations in this section are expressed in $L^2$-norm, and we convert the relations to $L^1$-norm in the next section.  

\subsection{Poincar\'e-type Inequality}
We start our derivation by showing a Poincar\'e-type inequality for the solution of a parabolic equation~\cref{eq:parabolic}, which bounds a function by its derivatives and the geometry of its domain. The standard Poincar\'e's inequality is of the form $\norm{u-\bar u}_{L^2(D)} \le C \norm{\nabla u}_{L^2(D)}$, where $\bar u$ is the average of $u$ in $D$, and $C$ is some positive constant. In~\cref{eq1419w} we find that when $D$ is convex, a closed-form expression for the constant $C$ can be derived. 

\begin{lemma} \label{lem01}
Let $D$ be a bounded convex domain in $\R^n$ and let $u \in W^{1,1}(D) = \{f \in L^1(D)  :  \partial_x f \in L^1(D) \}$ where $\partial_x$ is taken in the weak sense. Let $\eta$ be a nonnegative function such that $\int_D \eta\d y >0$ and $0 \le \eta(y) \le 1$. Then, for $x\in D$, we have
\begin{equation}				\label{eq1418w}
\abs{u(x)-\bar{u}_\eta} \le \frac{d^n}{n \int_D \eta\d y} \int_D \abs{x-y}^{1-n}\, \abs{\nabla u(y)}\d y,
\end{equation}
where $\bar{u}_\eta= \int_D u(y) \eta(y) \d y/\int_D \eta(y)\d y$ and $d=\diam D.$
In particular, we have
\begin{equation}				\label{eq1419w}
\norm{u-\bar{u}_\eta}_{L^2(D)} \le \frac{\omega_n^{1-\frac{1}{n}} \abs{D}^{\frac1n}}{\int_D \eta\d y} d^n\, \norm{\nabla u}_{L^2(D)},
\end{equation}
where $\omega_n$ denotes the volume of the unit ball in $\R^n$.
\end{lemma}
\begin{proof}
From the Fundamental Theorem of Calculus,
\[u(x)-u(y)=-\int_0^{\abs{x-y}} \partial_r u(x+r\omega)\d r,\quad \omega=\frac{y-x}{\abs{y-x}}, \quad x,y\in D.\]
By multiplying by $\eta(y)$ on both sides and integrating with respect to $y$ over $D$, we obtain
\[(u(x)-\bar{u}_\eta) \int_D \eta(y)\d y=-\int_D \eta(y)\int_0^{\abs{x-y}} \partial_r u(x+r\omega)\d r \d y.\]
For $x\in \R^n$, we define the function
\[ V(x+r\omega) = \begin{cases} 
      \abs{\partial_r u(x+r\omega)}, & \text{if } x+r\omega \in D,\\
      0, & \text{otherwise}.
   \end{cases}
\]
Then, we have
\begin{align*}
\abs{u(x)-\bar{u}_\eta} &\le \frac{1}{\int_D \eta\d y} \int_D \int_0^\infty V(x+r\omega) \d r  \d y\\
&= \frac{1}{\int_D \eta\d y} \int_0^\infty \int_{\abs{\omega}=1} \int_0^d V(x+r\omega)\rho^{n-1} \d\rho \d\omega \d r \\
&= \frac{d^n}{n \int_D \eta\d y} \int_0^\infty \int_{\abs{\omega}=1} V(x+r\omega) \d\omega \d r = \frac{d^n}{n \int_D \eta\d y} \int_D \abs{x-y}^{1-n}\, \abs{\nabla u(y)}\d y.
\end{align*}
Note that we have used the fact that $\eta\leq 1$ to obtain the first inequality. Finally, \cref{eq1419w} is obtained by applying~\citep[Lem.~7.12]{gilbarg2001elliptic} to \cref{eq1418w}, with $p=q=2$ and $\mu = 1/n$.
\end{proof}
It is worth noting the differences between the standard Poincar\'{e} inequality and the one in~\cref{lem01} as the average of $u$ is replaced by a ``weighted'' average of $u$. This is important for deriving~\cref{lem02}.

The importance of this lemma is that one can bound the $L^2$-norm of $u$ minus some constant by the $L^2$ norm of the spatial gradient (not the full gradient) of $u$ when $u$ is a solution of $\mathcal{P}u=0$. It is complementary to the parabolic Caccioppoli's inequality (see~\cref{lem03}), where the spatial gradient of $u$ is controlled by $u$ itself.
\begin{lemma} \label{lem02}
Let $\Omega \subset \R^n$ be a domain and $D\subset \R^n$ be a bounded convex set such that $\Omega \cap D \neq \emptyset$.
Suppose there is a constant $\theta \in (0,1)$ such that one of the following holds:
\begin{enumerate}
\item $\abs{D\setminus \Omega} \ge \theta \abs{D}$.
\item There exists a ball $B \subset \Omega \cap D$ such that $\abs{B} \ge \theta \abs{D}$.
\end{enumerate}
Then, for any $u$ satisfying $\mathcal{P}u=0$ in $Q=(\Omega \cap D) \times I$, where $I$ is an open bounded interval of $\R^+$, and $u=0$ on $(\partial \Omega \cap D) \times I$, there exists a constant $c\in\R$ such that
\begin{equation} \label{eq1619w}
\norm{u-c}_{L^2(Q)} \le \left( \frac{2^{2n}}{\theta^2} \left(\frac{\omega_n}{\abs{D}} \right)^{2-\frac2n} d^{2n} +    \frac{2^{2n+3}\Lambda^2 \omega_n^{\frac2n} \abs{I}^2}{\theta^{2+\frac2n} \abs{D}^{\frac2n}} \right)^{1/2}\norm{\nabla u}_{L^2(Q)},
\end{equation}
where $d$ and $\omega_n$ are defined in \cref{lem01}, and $\Lambda$ is a constant related to uniform parabolicity defined in \cref{eq_uni_parab}.
\end{lemma}

\begin{proof}
We denote $\tilde{Q}=D \times I$ and extend $u$ by zero on $\tilde Q \setminus Q$ so that $u$ is defined on $\tilde{Q}$. The proof is done in two steps assuming one of the conditions.

\begin{enumerate}[wide,labelwidth=0pt]
\item We first consider the case when $\abs{D\setminus \Omega} \ge \theta \abs{D}$. If $\eta$ is the indicator function of $D\setminus \Omega$, then \cref{lem01} yields
\[\int_{D} \abs{u(x,t)}^2 \d x \le \frac{\omega_n^{2-\frac{2}{n}} \abs{D}^{\frac2n}}{\abs{D\setminus \Omega}^2} d^{2n}\, \int_{D} \abs{\nabla u(x,t)}^2\d x,\quad t \in I.\]
By integrating with respect to $t \in I$ and using $\abs{D\setminus \Omega} \geq \theta \abs{D}$, we find that
\begin{equation} \label{eq_ineq_lemma_3}
\int_I \int_{D} \abs{u(x,t)}^2 \d x \d t \leq \frac{\omega_n^{2-\frac{2}{n}} \abs{D}^{\frac2n}}{\theta^2\abs{D}^2} d^{2n}\int_I \int_{D} \abs{\nabla u(x,t)}^2\d x\d t.
\end{equation}
Since $u=0$ in $\tilde Q \setminus Q$, \eqref{eq_ineq_lemma_3} implies \cref{eq1619w} with $c=0$.

\item Next, we consider the case when there exists a ball $B=B(x_0,r) \subset \Omega \cap D$ such that $\abs{B} \ge \theta \abs{D}$. Let $\eta$ be a smooth function such that
\begin{equation}				\label{eq1057th}
0\le \eta \le 1, \quad\supp \eta \subset B,\quad \abs{\nabla \eta} \le \frac{2}{r},\quad\text{and}\quad \int_B \eta(x)\d x \ge \frac{1}{2^n} \abs{B}.
\end{equation}
We denote
\[
\bar u_\eta(t)\coloneqq\frac{1}{\int_D \eta\d x} \int_D u(x,t)\eta(x)\d x,\quad \hat u_\eta \coloneqq \frac{1}{\abs{I}} \int_I \bar u_\eta(t)\d t=\frac{1}{\int_{\tilde Q} \eta\d x\d t}\int_{\tilde Q} u(x,t)\eta(x)\d x\d t.
\]
Then, we have
\begin{multline}	 \label{eq1741sat}
\int_{\tilde Q}\, \abs{u(x,t)-\hat u_\eta}^2\d x\d t  \le 2 \int_{\tilde Q} \,\abs{u(x,t)-\bar u_\eta(t)}^2 + \abs{\bar u_\eta(t)-\hat u_\eta}^2\d x\d t\\
\le \frac{2^{2n}}{\theta^2} \left(\frac{\omega_n}{\abs{D}} \right)^{2-\frac2n} d^{2n} \int_I \int_D \abs{\nabla u(x,t)}^2\d x\d t+ 2\abs{D} \int_I \abs{\bar u_\eta(t)-\hat u_\eta}^2\d t,
\end{multline}
where the second inequality follows from \cref{eq1419w,eq1057th}, and the assumption that $\abs{B} \ge \theta \abs{D}$.

We now bound the second term in the right-hand side of \cref{eq1741sat}. Since $\mathcal{P}u=\partial_t u-\nabla\cdot(A\nabla u)=0$ in $Q=(\Omega \cap D) \times I$, $u=0$ on $(\partial \Omega \cap D) \times I$, and $\supp \eta \subset B \subset \Omega\cap D$, we multiply the equation $\mathcal{P}u=0$ by $\eta$ and integrate by parts over $\Omega \cap D$ to obtain
\[
\int_{\Omega\cap D} \frac{\partial}{\partial t} u(x,t) \eta(x)\d x+ \int_{\Omega \cap D} A(x,t) \nabla u(x,t) \cdot \nabla \eta(x)\d x=0.
\]
Then, integrating over $t \in [t_0, t_1]\subset I$ yields
\begin{equation}\label{eq1115th}
\begin{aligned}
\left|\int_{t_0}^{t_1}\left(\frac{d}{dt}\int_{\Omega \cap D}  u(x,t) \eta(x)\d x\right)\d t\right| &= \left|\int_{t_0}^{t_1}\int_{\Omega \cap D} A(x,t) \nabla u(x,t) \cdot \nabla \eta(x) \d x\d t\right| \\
 &\le \int_{t_0}^{t_1}\int_{\Omega \cap D} \abs{A(x,t) \nabla u(x,t) \cdot \nabla \eta(x)} \d x\d t\\
&\le \frac{2 \Lambda}{r} \int_{Q} \abs{\nabla u}\d x\d t= 2 \Lambda \left(\frac{\omega_n}{\abs{B}}\right)^{\frac{1}{n}}  \int_{\tilde Q} \abs{\nabla u}\d x\d t,
\end{aligned}
\end{equation}
where the second inequality follows from \cref{eq1057th} and the uniform parabolicity condition. We now use the fact that $\int_B \eta\d x \ge \frac{1}{2^n} \abs{B}$ to obtain
\[\left|\int_{t_0}^{t_1}\left(\frac{d}{dt}\int_{\Omega \cap D}  u(x,t) \eta(x)\d x\right)\d t\right|= \left|\int_B\eta\d x\int_{t_0}^{t_1}\frac{d}{dt}\bar{u}_\eta(t)\d t\right|\ge \frac{\abs{B}}{2^n}\abs{\bar u_\eta(t_1)-\bar u_\eta(t_0)},\]
which when combined with \cref{eq1115th}, gives the following inequality:
\[ \frac{\abs{B}}{2^n} \abs{\bar u_\eta(t_1)-\bar u_\eta(t_0)} \le 2 \Lambda \left(\frac{\omega_n}{\abs{B}}\right)^{\frac{1}{n}}  \int_{\tilde Q} \abs{\nabla u}\d x\d t. \]
Therefore, we have
\begin{equation}			\label{eq1743sat}
\abs{\bar u_\eta(t_1)-\bar u_\eta(t_0)} \le \frac{2^{n+1}\Lambda \omega_n^{\frac1n}}{\abs{B}^{1+\frac1n}} \int_{\tilde Q} \abs{\nabla u} \d x\d t\le \frac{2^{n+1}\Lambda \omega_n^{\frac1n}}{\theta^{1+\frac1n}\abs{D}^{1+\frac1n}} \int_{\tilde Q} \abs{\nabla u}\d x\d t,
\end{equation}
as $|B|\geq \theta |D|$. Moreover, we have,
\[\bar u_\eta(t)-\hat u_\eta= \frac{1}{\abs{I}} \int_{I}\bar u_\eta(t)- \bar u_\eta(s) \d s,\quad t\in I\]
and, using \cref{eq1743sat}, we deduce that
\[2\abs{D} \int_I \abs{\bar u_\eta(t)-\hat u_\eta}^2\d t\leq  2 \abs{\tilde Q} \left(\frac{2^{n+1}\Lambda \omega_n^{\frac1n}}{\theta^{1+\frac1n}\abs{D}^{1+\frac1n}} \int_{\tilde Q} \abs{\nabla u}\d x \d t \right)^2.\]
We combine this equation with \cref{eq1741sat} to obtain
\[\int_{\tilde Q} \abs{u-\bar u_\eta}^2 \d x \d t \le \left( \frac{2^{2n}}{\theta^2} \left(\frac{\omega_n}{\abs{D}} \right)^{2-\frac2n} d^{2n} +    \frac{2^{2n+3}\Lambda^2 \omega_n^{\frac2n} \abs{I}^2}{\theta^{2+\frac2n} \abs{D}^{\frac2n}} \right) \int_{\tilde Q} \abs{\nabla u}^2\d x \d t,\]
and choose the constant $c=\bar{u}_\eta$ to conclude the proof of \cref{eq1619w}.
\end{enumerate}
\end{proof}

If the diameter $d$ of the spatial domain $D$ satisfies $d \simeq \abs{I}^{\frac12} \simeq \abs{D}^{\frac1n}$, where $|D| \coloneqq \int_D \d x$ denotes the volume of $D$, then \cref{eq1619w} can be simplified to
\[\norm{u-c}_{L^2(Q)} \lesssim  d \norm{\nabla u}_{L^2(Q)},\]
where $\lesssim$ denotes an inequality up to the multiplication of the right-hand side by a constant. For example, this situation arises when $D\times I= Q_r^-(X_0) \coloneqq B_r(x_0) \times (t_0-r^2, t_0]$, for a given $t_0>0$. The assumptions of~\cref{lem02} can be shown to hold in some simple contexts. For example, if $\Omega$ satisfies the uniform interior cone condition~\citep[Chapt.~7.7]{gilbarg2001elliptic} and $D$ is such that $\abs{D} \simeq d^n$, such as a cube or a ball, then there exists a constant $\theta \in (0,1)$ and $\delta>0$ depending on $\Omega$, such that if $d \le \delta$, one of the conditions of \cref{lem02} holds.

\subsection{Cacciopolli's Inequality}

Next, we show Cacciopolli's inequality for solutions of parabolic equations~\eqref{eq:parabolic}, which can also be seen as a kind of reverse Poincar\'e inequality. In particular, it says that the $L^2$ norm of the spatial gradient (not the full gradient) of $u$ can be bounded above by the $L^2$ norm of $u$ on a slightly larger domain.

\begin{lemma} \label{lem03}
Let $D$ be a domain such that $D \cap \Omega \neq \emptyset$, $\Gamma\coloneqq \partial D \cap \overline \Omega$, and $K\subset D$ be a domain such that $\delta_0\coloneqq\dist_{L^2(D)}(K, \Gamma)>0$.
Let $I=(t_0, t_1)\subset\R$ and $I'=(t_0+\delta_1, t_1)$, for a given $0<\delta_1<t_1-t_0$.
Then, for any $u$ satisfying $\mathcal{P}u=0$ in $Q\coloneqq (D\cap\Omega)\times I$ and $u=0$ on $(\partial \Omega \cap D) \times I$, we have
\[ \int_{I'}\int_{K\cap \Omega} \abs{\nabla u(x,t)}^2\d x\d t
 \le \left(\frac{4\Lambda^2}{\lambda^2\delta_0^2}+\frac{2}{\lambda \delta_1} \right)\|u\|_{L^2(Q)}^2.\]
\end{lemma}
\begin{proof}
Let $\eta\in C^1(\R^n)$ such that $0\le \eta \le 1$, $\eta=1$ on $K$, $\eta=0$ in a neighbourhood of $\Gamma$, and $\abs{\nabla \eta} \le 1/\delta_0$. We also consider a similar function $\zeta \in C^1(\R)$, defined on the temporal domain, such that $0\le \zeta \le 1$, $\zeta=1$ on $I'$, $\zeta=0$ in a neighbourhood of $t_0$, and $\abs{\zeta'} \le 1/\delta_1$.

Then, the function $\eta^2 \zeta^2 u$, defined on $(D\cap \Omega) \times (t_0,t)$ for $t_0<t<t_1$, satisfies the equation $\mathcal{P}(\eta^2 \zeta^2 u) = 0$. After multiplying the equation by $u$, integrating by parts over $(D\cap \Omega) \times (t_0,t)$, and using the fact that $\eta$ vanishes near $\Gamma$, which gives $\eta^2\zeta^2 u=0$ on $\partial(D\cap \Omega)$, we obtain,
\begin{equation} \label{eq_lemm_4}
\int_{t_0}^t \int_{D\cap \Omega} \eta^2(x) \zeta^2(t) u(x,t) \frac{\partial}{\partial t}u(x,t)\d x\d t + \int_{t_0}^t \int_{D\cap\Omega} A(x,t) \nabla u\cdot \nabla(\eta^2 \zeta^2 u)\d x\d t=0.
\end{equation}
The first term in \cref{eq_lemm_4} can be reformulated as
\begin{align*}
\int_{t_0}^t \int_{D\cap \Omega} \eta^2(x) \zeta^2(t) u\frac{\partial}{\partial t}u\d x\d t &= \int_{t_0}^t\int_{D\cap \Omega} \frac{\partial}{\partial t} \left(\frac12 \eta^2(x) \zeta^2(t) u^2\right)- \eta^2(x) \zeta(t) \zeta'(t) u^2\d x\d t.\\
&=\int_{D\cap \Omega} \frac12 \eta^2(x) \zeta^2(t) u(x,t)^2\d x- \int_{t_0}^t\int_{D\cap \Omega}\eta^2 \zeta \zeta' u^2\d x\d t,
\end{align*}
since $\zeta$ vanishes near $t_0$.
Then, the second term of \cref{eq_lemm_4} satisfies
\[\int_{t_0}^t \int_{D\cap\Omega} A(x,t) \nabla u\cdot \nabla(\eta^2 \zeta^2 u)\d x\d t=\int_{t_0}^t \int_{D\cap\Omega} \eta^2 \zeta^2 A \nabla u \cdot \nabla u + 2\eta \zeta^2 u A \nabla u \cdot \nabla \eta \d x\d t.
\]
Combining these two terms gives the following equality:
\begin{align*}
\int_{D\cap \Omega} \frac12 \eta^2 \zeta^2(t) u^2\d x + \int_{t_0}^t \int_{D\cap\Omega} \eta^2 \zeta^2 A \nabla u \cdot \nabla u &\leq \int_{t_0}^t \int_{D\cap\Omega} 2\eta \zeta^2 |u| |A \nabla u \cdot \nabla \eta| \d x\d t\\
&+\int_{t_0}^t\int_{D\cap \Omega} \eta^2 \zeta \abs{\zeta'} u^2\d x\d t.
\end{align*}
We now use the uniform parabolicity of $\mathcal{P}$: $\lambda|\nabla u|^2\leq A\nabla u\cdot\nabla u$ and $|A\nabla u\cdot\nabla \eta|\leq \Lambda|\nabla u\|\nabla \eta|$ to obtain
\begin{multline*}	
\int_{D\cap \Omega} \frac12 \eta^2 \zeta^2(t) u^2\d x+ \lambda \int_{t_0}^t\int_{D\cap \Omega} \eta^2 \zeta^2 \abs{\nabla u}^2 \d x\d t\\
\leq  2 \Lambda \int_{t_0}^t\int_{D\cap \Omega} \eta \zeta^2 \abs{u} \abs{\nabla u} \abs{\nabla \eta} \d x\d t + \int_{t_0}^t\int_{D\cap \Omega} \eta^2 \zeta \abs{\zeta'} u^2\d x\d t\\
\le \epsilon \int_{t_0}^t\int_{D\cap \Omega} \eta^2 \zeta^2 \abs{\nabla u}^2\d x\d t + \frac{\Lambda^2}{\epsilon}\int_{t_0}^t\int_{D\cap \Omega} \zeta^2 \abs{\nabla \eta}^2u^2\d x\d t+\int_{t_0}^t\int_{D\cap \Omega} \eta^2 \zeta \abs{\zeta'} u^2\d x\d t,
\end{multline*}
where the second inequality follows from Young's inequality: $ab\leq (a^2+b^2)/2$ with $a=\sqrt{2\epsilon}\eta|\nabla u|$ and $b=\Lambda\sqrt{2/\epsilon}|u\|\nabla \eta|$, for any $\epsilon>0$. Then, choosing $\epsilon= \lambda/2$ and using the properties $\abs{\nabla \eta} \le 1/\delta_0$, $\abs{\zeta'} \le 1/\delta_1$, $0 \le \eta \le 1$, and $0 \le \zeta \le 1$, we find that
\begin{equation}						\label{eq1927f}
\frac12 \int_{D\cap \Omega} \eta^2 \zeta^2(t) u^2\d x+ \frac{\lambda}{2} \int_{t_0}^t\int_{D\cap \Omega} \eta^2 \zeta^2 \abs{\nabla u}^2\d x\d t \leq 
\left(\frac{2\Lambda^2}{\lambda \delta_0^2}+\frac{1}{\delta_1}\right) \int_{t_0}^t\int_{D\cap \Omega} u^2 \d x\d t.
\end{equation}
The result follows from the properties of $\eta$ and $\zeta$.
\end{proof}

Since $t\in (t_0, t_1)$ is arbitrary in \cref{eq1927f}, we also obtain the following inequality:
\[
\esssup_{t \in I'} \int_{K\cap \Omega} \abs{u(x,t)}^2\d x
 \leq \left(\frac{4\Lambda^2}{\lambda \delta_0^2}+\frac{2}{\delta_1} \right) \int_I \int_{D\cap \Omega} \abs{u(x,t)}^2\d x \d t,
\]
which can be combined with \cref{lem03} to find that
\begin{equation} \label{eq2047th}
\esssup_{t \in I'} \int_{K\cap \Omega} \abs{u(x,t)}^2\d x + \int_{I'}\int_{K\cap \Omega} \abs{\nabla u(x,t)}^2\d x\d t 
 \leq C  \int_I \int_{D\cap \Omega} \abs{u(x,t)}^2\d x \d t,
\end{equation}
where $C$ is a constant depending on $\lambda$, $\Lambda$, $K$, and $I'$. From the definition of the Banach space $V_2$,~\cref{eq2047th} can be understood as an upper bound of the norm of solution $u$ with respect to this Banach space. Using the bound given by \cref{eq2047th}, we can introduce a subspace $\mathscr{X}(D \times I)$ of $L^2(D \times I)$ and prove its closeness.

\begin{lemma} \label{lem04}
Let $D$ be a domain such that $D \cap \Omega \neq \emptyset$,  $\Gamma= \partial D \cap \overline \Omega$, $I = (t_0, t_1)$. Define $\mathscr{X}(D \times I)$ to be the subspace of $L^2(D \times I)$ consisting of the functions $u$ satisfying the following conditions.
\begin{enumerate}
\item $u=0$ on $(D \setminus \Omega) \times (t_0, t_1)$.
\item $u \in V_2(Q')$ for all $Q'=K \times (t_0+\tau, t_1)$, where $K \subset D$ with $\dist(K, \Gamma)>0$ and $0<\tau <t_1-t_0$.
\item $\mathcal{P}u=0$ in $(D\cap \Omega) \times (t_0, t_1)$ in the sense that, for almost all $t \in (t_0, t_1)$, the equality
\begin{equation} \label{eq:lem4_cond3}
\int_{D\cap \Omega} u(x,t)\eta(x,t)\d x - \int_0^{t}\int_{D\cap \Omega} u \partial_t \eta\d x\d t+\int_0^{t}\int_{D\cap \Omega} A\nabla u \cdot \nabla \eta\d x\d t=0
\end{equation}
holds for all smooth test function $\eta$ vanishing near $\partial(D\cap\Omega) \times (t_0, t_1)$ and $\overline{D\cap \Omega}\times \set{t_0}$.
\end{enumerate}
Then, the space $\mathscr{X}(D \times I)$ is closed in $L^2(D \times I)$.
\end{lemma}

\begin{proof}
Let $v\in L^2(D \times I)$ and $\set{v_k}_{k \in \N} \subset \mathscr{X}(D \times I)$ be a sequence converging to $v$, we want to show that $v\in \mathscr{X}(D \times I)$. First, using \cref{eq2047th}, we have
\[ \|v_k\|_{V_2(Q')} \le C \|v_k\|_{L^2(D \times I)}.\]
Following the Banach--Alaoglu Theorem, there exists a subsequence $\set{v_{i_k}}_{k \in \N}$ of $\set{v_k}_{k \in \N}$ that converges weakly in $V_2(Q')$ to $\hat{v} \in V_2(Q')$. Therefore, for any $w \in L^2(Q')$, we have $\langle v,w\rangle_{L^2(Q')} = \lim_{k \rightarrow \infty} \langle v_{i_k},w\rangle_{L^2(Q')} = \langle\hat{v},w\rangle_{L^2(Q')}$, which shows that $v = \hat{v} \in V_2(D \times I)$. By the same argument $v$ satisfies \cref{eq:lem4_cond3}. Finally, $v_k = 0$ on $(D \setminus \Omega) \times (t_0, t_1)$ implies that $v = 0$ on $(D \setminus \Omega) \times (t_0, t_1)$, and $v \in \mathscr{X}(D \times I)$.
\end{proof}

\subsection{Separable Approximation}

The following two lemmas quantify the dimension of a finite-dimensional subspace $W$ needed to approximate a function in $\mathscr{X}(D \times I)$ up to a prescribed relative error. In this way, given a parabolic equation solution and the desired accuracy, we can determine the rank of the separable approximant of the corresponding Green's function. We begin by defining a finite-dimensional subspace $V_k$ of the space $\mathscr{X}(D \times I)$, and bounding the distance between parabolic equation solutions and $V_k$.

\begin{lemma} \label{lem05}
Let $\Omega \subset \R^n$ be a domain satisfying the uniform interior cone condition, $D=\set{ x \in \R^n: \norm{x-x_0}_\infty < d/2}$ be a cube in $\R^n$ of side length $d>0$ centered at $x_0\in\R^n$ such that $\Omega \cap D \neq \emptyset$, and $I=(t_0, t_0+\beta d^2)$ for some constant $\beta>0$ and $t_0>0$. Then there exists $\delta_0>0$ depending only on $\Omega$ such that for any $k \ge (1+ \lceil{d/\delta_0} \rceil)^{n+2}$, there is a subspace $V_k \subset \mathscr{X}(D \times I)$ with $\dim V_k \le k$ such that
\begin{equation} \label{eq2154sun}
\dist_{L^2(D \times I)}(u, V_k) \le c_{{\rm appr}}\,  k^{-\frac{1}{n+2}} d \,\norm{\nabla u}_{L^2(D \times I)},\quad u \in \mathscr{X}(D \times I) \cap V_2(D \times I),
\end{equation}
where $\mathscr{X}(D \times I)$ is defined by \cref{lem04}, $c_{{\rm appr}} = 2^{n+2} (\omega_n^{2-2/n}\theta^{-2}+2\Lambda^2 \omega_n^{2/n} \beta^2 \theta^{-2-2/n})^{1/2}$, and $\theta=\theta(\Omega)$ is determined by the characteristics of the uniform cone condition.
\end{lemma}
\begin{proof}
Let $\ell\geq 1$, we first subdivide the cube $D$ uniformly into $\ell^n$ sub-cubes and subdivide the interval $(t_0, t_1)$ into $\ell^2$ subintervals to form $\ell^{n+2}$ cylinders of the form $Q_i=D_i\times I_i$, where $D_i$ is a cube of side length $d/\ell$ and $I_i$ is an interval of length $\beta (d/\ell)^2$, for $1\leq i\leq \ell^{n+2}$.
Since $\Omega$ satisfies the uniform interior cone condition, there exists $\delta_0=\delta_0(\Omega)>0$ and $\theta=\theta(\Omega)>0$ such that if $d/\ell \le \delta_0$, then either of the conditions 1 or 2 in \cref{lem02} holds for all $D_i$ satisfying $D_i \cap \Omega \neq \emptyset$.

We now choose $\ell \ge \lceil{d/\delta_0} \rceil$ so that $d/\ell \le \delta_0$, and define the space $W$ of piecewise constant functions on the domains $Q_i$ by
\[
W\coloneqq \set{v \in L^2(D \times I): v \text{ is constant on }Q_i \text{ for all }1\leq i\leq \ell^{n+2}},
\]
then $\dim W = \ell^{n+2}$. Let $1\leq i\leq \ell^{n+2}$, we first consider the case where $D_i \cap \Omega \neq \emptyset$. We know that $D_i$ is convex, $\abs{D_i}=(d/\ell)^n$, $\diam(D_i)= \sqrt[n]{2}d/\ell$, and $\abs{I_i}=\beta d^2/\ell^2$. According to \cref{lem02}, there exists a constant $c_i\in \R$ such that
\begin{equation} \label{eq2115m}
\int_{Q_i} \abs{u-c_i}^2\d x\d t \leq \left( \frac{2^{2n+2}\omega_n^{2-\frac2n}}{\theta^2}+\frac{2^{2n+3}\Lambda^2 \omega_n^{\frac2n} \beta^2 }{\theta^{2+\frac2n}} \right) \frac{d^2}{\ell^2} \int_{Q_i} \abs{\nabla u}^2\d x\d t.
\end{equation}
If $D_i \cap \Omega =\emptyset$, then $u \in \mathscr{X}(D \times I)$ implies that $u=0$ on $Q_i$, and \cref{eq2115m} holds with $c_i=0$.

Next, we define a piecewise constant function $\bar u\in W$ such that $\bar u\vert_{Q_i}=c_i$ for $1\leq i\leq \ell^{n+2}$. Summing \cref{eq2115m} over $i$ yields the following inequality
\[\norm{u-\bar u}_{L^2(D \times I)} \le \left( \frac{2^{2n+2}\omega_n^{2-\frac2n}}{\theta^2}+\frac{2^{2n+3}\Lambda^2 \omega_n^{\frac2n} \beta^2 }{\theta^{2+\frac2n}}\right)^{1/2} \frac{d}{\ell} \norm{\nabla u}_{L^2(D \times I)}.\]
Let $k \geq (1+ \lceil{d/\delta_0} \rceil)^{n+2}$ be an integer and choose $\ell= \lfloor{k^{\frac{1}{n+2}}} \rfloor$ such that $\ell^{n+2} \le k <(\ell+1)^{n+2}$, $\ell \ge \lceil{d/\delta_0}\rceil$, and $\dim W \le \ell^{n+2} \le k$. Now, since $1/\ell \le 2/(\ell+1) \le 2 k^{-\frac{1}{n+2}}$, we have
\begin{equation} \label{eq:lm5_pb}
\norm{u-\bar u}_{L^2(D \times I)} \le \left( \frac{2^{2n+4}\omega_n^{2-\frac2n}}{\theta^2}+\frac{2^{2n+5}\Lambda^2  \omega_n^{\frac2n} \beta^2 }{\theta^{2+\frac2n}}\right)^{1/2} k^{-\frac{1}{n+2}} d \norm{\nabla u}_{L^2(D \times I)}.
\end{equation}
Finally, let $P: L^2(D \times I) \to \mathscr{X}(D \times I)$ be the $L^2(D \times I)$-orthogonal projection onto $\mathscr{X}(D \times I)$ and $V_k\coloneqq P(W)$. The statement of the lemma follows from \cref{eq:lm5_pb} and
\[
\dist_{L^2(D \times I)}(u, V_k) \leq \norm{u-P(\bar u)}_{L^2(D \times I)} = \norm{P(u- \bar u)}_{L^2(D \times I)} \le \norm{u-\bar u}_{L^2(D \times I)}.
\]
\end{proof}

From~\cref{lem05}, we can use the constant $\delta_0$ to fix an accuracy and construct a finite-dimensional subspace $W$ of $\mathscr{X}(D \times I)$ such that the $L^2(D \times I)$-distance between solutions to \cref{eq:parabolic} and $W$ is within the accuracy threshold. In the following lemma, we provide an upper bound on the dimensionality of $W$.

\begin{lemma} \label{lem06}
Let $\Omega \subset \R^n$ be a domain satisfying the uniform interior cone condition, $D_1=\set{ x \in \R^n: \norm{x-x_0}_\infty < d/2}$, $D=\set{ x \in \R^n: \norm{x-x_0}_\infty <(1/2+\rho)d}$, $I_1=(t_0, t_0+\beta d^2)$, and $I=(t_0, t_0+\beta(1+2\rho)^2 d^2)$, for some $\beta, \rho>0$. Assume that $\Omega \cap D_1 \neq \emptyset$ and let $\delta_0$ and $\theta$ be the constant introduced in \cref{lem05} characterized by the uniform cone condition. Then, for any $M \geq \exp\{2(\lceil{(1+2\rho)d/\delta_0}\rceil +1)\}$, there exists a subspace $W \subset \mathscr{X}(D_1\times I_1)$ such that
\[\dist_{L^2(D_1\times I_1)}(u, W) \le \frac{1}{M} \,\norm{u}_{L^2(D\times I)},\quad \forall u \in \mathscr{X}(D\times I),\]
and
\begin{equation} \label{eq_cp}
\dim(W) \le c_\rho^{n+2} \lceil{\log M}\rceil^{n+3} + \lceil{\log M}\rceil,\quad c_\rho=e(2+\rho^{-1})\kappa_c \,c_{{\rm appr}},
\end{equation}
where $\kappa_c=\sqrt{4\Lambda^2/\lambda^2+1/(2\lambda \beta)}$ and $c_{{\rm appr}}$ is the constant defined in \cref{lem05}.
\end{lemma}

\begin{proof}
Our proof follows closely the argument for elliptic PDEs~\citep[Lem.~2.6]{bebendorf2003existence}. Let $i \in \N_{\geq 1}$ and, for $0\leq k\leq i$, define
\begin{align*}
D^{(k)}&=\set{x \in \R^n: \norm{x-x_0}_\infty < (1/2+(1-k/i)\rho)d},\\
I^{(k)}&=(t_0, t_0+\beta (1+2(1-k/i)\rho)^2d^2),
\end{align*}
such that $D_1=D^{(i)}\subset D^{(i-1)} \subset \cdots \subset D^{(0)} = D$, and
$I_1=I^{(i)} \subset I^{(i-1)} \subset \cdots \subset I^{(0)} = I$. We also denote $Q^{(k)}=D^{(k)}\times I^{(k)}$ and $\mathscr{X}^{(k)}= \mathscr{X}(Q^{(k)})$. Let $1\leq j\leq i$. By applying \cref{lem03} with the domains $K=D^{(j)}$, $D=D^{(j-1)}$, $I'=I^{(j)}$, and $I=I^{(j-1)}$, we find that
\begin{equation} \label{eq_lemma_8_1}
\norm{\nabla v}_{L^2(Q^{(j)})} \le \kappa_c \frac{i}{\rho d}\, \norm{v}_{L^2(Q^{(j-1)})},\quad v \in \mathscr{X}^{(j-1)}.
\end{equation}
Moreover, \cref{eq2047th} shows that $\mathscr{X}^{(j-1)}\subset \mathscr{X}^{(j)} \cap V_2(Q^{(j)})$. In addition, the choice of $D=D^{(j)}$ and $I=I^{(j)}$ in \cref{lem05}, shows that there exists a subspace $V_j \subset \mathscr{X}^{(j)}$ such that
\begin{equation} \label{eq_lemma_8_2}
\dist_{L^2(Q^{(j)})}(v, V_j) \le c_{{\rm appr}}\, \frac{(1+2\rho)d}{iB} \,\norm{\nabla v}_{L^2(Q^{(j)})},\quad  v \in \mathscr{X}^{(j)} \cap V_2(Q^{(j)}),
\end{equation}
where $B$ is a constant chosen so that $\lceil{(1+2\rho)d/\delta_0}\rceil +1 \le iB \le k^{\frac{1}{n+2}}$ and $\dim V_j \le k$. In particular, we can set $k\coloneqq\lceil{(iB)^{n+2}} \rceil$ so that $k \ge (1+ \lceil{(1+2\rho)d/\delta_0} \rceil)^{n+2}$. Combining \cref{eq_lemma_8_1,eq_lemma_8_2} yields
\[\dist_{L^2(Q^{(j)})}(v, V_j) \le \frac{1+2\rho}{\rho} \,\frac{c_{{\rm appr}} \,\kappa_{c}}{B}  \,\norm{v}_{L^2(Q^{(j-1)})},\quad v \in \mathscr{X}^{(j-1)}.\]
We now choose $B\coloneqq B_0 M^{1/i}$ and $B_0\coloneqq c_{{\rm appr}} \kappa_{c} \frac{1+2\rho}{\rho}$ to obtain the following inequality:
\begin{equation}			\label{eq1831f}
\dist_{L^2(Q^{(j)})}(v, V_j) \le M^{-1/i}\,\norm{v}_{L^2(Q^{(j-1)})},\quad v \in \mathscr{X}^{(j-1)}.
\end{equation}

Now let $u \in \mathscr{X}^{(0)}$. We aim to iteratively express $u$ as a sum of functions in smaller subspaces. Initially, we define $v_0=u$ and use \cref{eq1831f} to decompose $v_0$ on $Q^{(1)}$ as $v_0\vert_{Q^{(1)}}=u_1+v_1$, where $u_1 \in V_1$ and $v_1$ satisfies
$\smash{\norm{v_1}_{L^2(Q^{(1)})}} \le \smash{M^{-1/i}\,\norm{v_0}_{L^2(Q^{(0)})}}$.
Consequently, we see that $v_1\in\mathscr{X}^{(1)}$. We can continue this process for $1\leq j\leq i$, such that $v_{j-1}\vert_{Q^{(j)}}=u_j+v_j$, where $u_j \in V_j$ and $v_j\in  \mathscr{X}^{(j)}$ satisfies 
$\smash{\norm{v_j}_{L^2(Q^{(j)})}} \le \smash{M^{-1/i}\,\norm{v_{j-1}}_{L^2(Q^{(j-1)})}}$. Finally, we define the subspace $\smash{W=\Span \set{ V_j \vert_{D_1\times I_1}: 1\leq j\leq i}}$ using the restrictions of the $V_j$ to the smallest domain $\smash{Q^{(i)}=D^{(i)}\times I^{(i)}=D_1\times I_1}$, which contain $\smash{u_j \vert_{D_1\times I_1} \in V_j \vert_{D_1\times I_1} \subset W}$ for $1\leq j\leq i$.
Therefore, the decomposition of $v_0$ as $\smash{v_0=v_i+\sum_{j=1}^i u_j}$ leads to
\[
\dist_{L^2(D_1\times I_1)}(v_0, W) \le \norm{v_i}_{L^2(D_1\times I_1)} \le (M^{-1/i})^i \norm{v_0}_{L^2(Q^{(0)})} = M^{-1} \norm{u}_{L^2(D\times I)}.
\]
We then choose $i=\lceil{\log M}\rceil$ and use the definition of $W$ to bound the dimension of $W$ by
\[
\dim(W) \le \sum_{j=1}^i \dim(V_j) = i \lceil{(iB)^{n+2}}\rceil \le i + B^{n+2} i^{n+3} \le \lceil{\log M}\rceil + B_0^{n+2} e^{n+2} \lceil{\log M}\rceil^{n+3},
\]
because $B=B_0 M^{1/i}\leq B_0 e$. The statement of the lemma follows by defining $c_\rho=B_0 e = c_{{\rm appr}} \kappa_{c} e \frac{1+2\rho}{\rho}$.
\end{proof}

We are now ready to prove that the Green's function associated with a parabolic PDE has a separable approximation in terms of $L^2$-norm on well-separated domains.

\begin{theorem} \label{thm01}
Let $\Omega \subset \R^n$ be a domain satisfying the uniform interior cone condition and $\rho>0$. Let $D_1,D_2\subset\R^n$ be two domains such that $D_1$ is convex and $I_1,I_2\subset (0,T)$ be two open bounded intervals, such that $Q_X=(D_1\cap\Omega)\times I_1$ and $Q_Y=(D_2\cap\Omega)\times I_2$ are admissible, i.e., $\dist(Q_X,Q_Y)\geq \rho \max\{\diam Q_X,\diam Q_Y\}$. Then, for any $\epsilon>0$ sufficiently small, there exists a separable approximation of the form
\[
G_k(x,t,y,s)=\sum_{i=1}^k u_i(x,t)  v_i(y,s),\quad (x,t,y,s)\in Q_X\times Q_Y,
\]
where $k \le k_\epsilon=c_{\rho/2}^{n+2} \lceil{\log\frac{1}{\epsilon}}\rceil^{n+3} + \lceil{\log \frac{1}{\epsilon}}\rceil$, and $c_\rho$ is defined in~\eqref{eq_cp}, such that
\begin{equation}			\label{eq1624sat}
\norm{G(\cdot,\cdot,y,s)-G_k(\cdot,\cdot,y,s)}_{L^2(Q_X)} \le \epsilon \norm{G(\cdot,\cdot,y,s)}_{L^2(\hat Q_X)},\quad (y,s) \in Q_Y,
\end{equation}
where $\hat{Q}_X\coloneqq \{X\in Q,\, \dist(X,Q_X)\leq \frac{\rho}{2}\diam Q_X\}$.
\end{theorem}
\begin{proof}
Let $\epsilon_0=e^{-2(\lceil{(1+\rho)d/\delta_0}\rceil +1)}$ with $\delta_0$ defined in~\cref{lem05}, $I_1 = (t_0-\beta d^2/2,t_0+\beta d^2/2)$, with $t_0>0$ and $\beta$ defined in \cref{eq_def_beta}, and $d=\max\{\diam Q_X,\diam Q_Y\}$. We also let $D=\{x\in \R^n,\,\dist(x,D_1\cap\Omega)\leq \frac{\rho d}{2}\}$ and $I=\{t_0-\beta d^2(1+\rho)^2/2,t_0+\beta d^2(1+\rho)^2/2\}$. Similarly to \cite{bebendorf2003existence}, we observe that because
$\dist(\hat{Q}_X,Q_Y)\geq \dist(D\times I,Q_Y)\geq \frac{\rho d}{2}>0$, the right-hand side $\norm{G(\cdot,\cdot,y,s)}_{L^2(\hat Q_X)}$ does not contain the singularity of $G$. 

According to~\cref{lem06}, with $M=\epsilon^{-1}$ and $\rho$ replaced with $\rho/2$, we can set $\set{u_1, \ldots, u_k}$ be the basis of the subspace $W \subset \mathscr{X}(D_1\times I_1)$ with $k=\dim W \le c_{\rho/2}^{n+2} \lceil{\log\frac{1}{\epsilon}}\rceil^{n+3} + \lceil{\log \frac{1}{\epsilon}}\rceil$. For any $(y,s) \in Q_Y$, the function $g_Y:=G(\cdot,\cdot,y,s)$ is in $\mathscr{X}(D\times I)$. Moreover, $g_Y=\hat g_Y+ r_Y$ holds with $\hat g_Y \in W$ and $\norm{r_Y}_{L^2(Q_X)} \le  \epsilon \norm{g_Y}_{L^2(\hat Q_X)}$. 
Then, expressing $\hat g_Y$ with the basis, we obtain $\smash{\hat g_Y=\sum_{i=1}^k v_i(y,s) u_i}$,
with coefficients $v_i(y,s)$ depending on $y$ and $s$. Since $(y,s)\in Q_Y$, the $v_i$ are functions defined on $Q_Y$. The function $G_k(x,t,y,s)=\sum_{i=1}^k u_i(x,t) v_i(y,s)$ satisfies the estimate \eqref{eq1624sat}.
\end{proof}
If we integrate~\cref{eq1624sat} over $(y,s)\in Q_Y$, we obtain the inequality stated in \cref{sec_main_result}:
\[
\norm{G-G_k}_{L^2(Q_X\times Q_Y)} \le \epsilon \norm{G}_{L^2(\hat{Q}_X\times Q_Y)}.
\]

\section{Learning Rate for Green's Functions Associated with Parabolic PDEs} \label{sec_learning_rate}
In this section, we combine \cref{thm01} and the generalization of the randomized SVD to HS operators~\citep{boulle2021learning,boulle2022generalization} to construct a global approximant of Green's functions associated with parabolic PDEs. We suppose that one can generate $N\geq 1$ arbitrary forcing terms $\{f_1,\ldots,f_N\}$ and observe the corresponding solutions $\{u_1,\ldots,u_N\}$ from an unknown parabolic PDE of the form of \cref{eq:parabolic} and its adjoint, and derive a learning rate by working out the number of training pairs needed to learn the Green's function within a prescribed tolerance. As discussed in \cref{sec_challenges}, there is an additional difficulty compared with the elliptic case~\citep{boulle2021learning} as Green's functions of parabolic operators are not guaranteed to be squared-integrable in spatial dimensions greater than two. Therefore, we prove the following theorem, which provides rigorous probability bounds for approximating the Green's function in the $L^1$-norm from a given number of forcing terms and solutions.

\begin{theorem} \label{thm_learning_rate}
Let $\Omega\subset\R^n$ be a domain satisfying the uniform interior cone condition, $\U=\Omega\times [0,T]$, $\epsilon>0$ sufficiently small, and $G$ be the Green's function associated with the parabolic operator $\mathcal{P}$ in \cref{eq:parabolic}. Then, there exists a randomized algorithm that can construct an approximation $\tilde{G}$ of $G$ using $\mathcal{O}(\epsilon^{-\frac{n+2}{2}}\log(1/\epsilon))$ many input-output pairs of \eqref{eq:parabolic} and its adjoint such that
\[\|G-\tilde{G}\|_{L^1(\U\times \U)} = \mathcal{O}(\Gamma_\epsilon^{-1/2}\epsilon)\,\|G\|_{L^1(\U\times \U)},\]
with probability $\geq 1-\mathcal{O}(\epsilon^{\log^{n+1}(1/\epsilon)})$. The quantity $\Gamma_\epsilon$ is defined in \cref{eq_gamma} and characterizes the quality of the training pairs to learn $G$.
\end{theorem}

The proof of the theorem occupies the rest of the section. It exploits the regularity result of Green's functions on admissible domains stated in \cref{thm01}, and standard Gaussian bounds near the diagonal $\mathcal{D}=\{(x,t,y,s)\in \U\times \U,\, (x,t)=(y,s)\}$.

\subsection{Hierarchical Partition of the Temporal Domain} \label{sec_hier}

We start the proof of \cref{thm_learning_rate} by constructing a hierarchical partition of the domain $\U\times \U$ into admissible and non-admissible domains. We aim to obtain a partition such that the vast majority of the subdomains are admissible, while the remaining non-admissible domains all have small areas. In this way, we can obtain accurate low-rank approximations on admissible domains by combining \cref{thm01} with the randomized SVD (see \cref{sec:rand_svd}), and safely neglect Green's functions on non-admissible domains. 

Without loss of generality, we assume that $\U=\Omega\times I$, where $\Omega\subset D=[0,1]^n$, $I=[0,1]$, and $\beta=1$ in \cref{eq_def_beta}. Otherwise it's straightforward to shift and scale $\Omega$ and $[0,T]$ by adjusting $\beta$. We construct a partition of $\U\times \U$ such that if $Q_X\times Q_Y$ is an element of the partition, where $Q_X = (\Omega\cap D_x)\times I_X$, then 
\begin{equation} \label{eq_size_partition}
\diam(D_X)=\diam(I_X)^{1/2}=\diam(Q_X)=\diam(Q_Y).
\end{equation}
This condition is determined by the metric defined in \cref{eq_def_beta} and guarantees that a large proportion of the domains in the partition are admissible. In this setting, choosing $\rho = 1$ in \cref{eq_adm_domain} is convenient for the definition of admissible sets.

First, we define a hierarchical partition of $D\times I$ for an arbitrary $n_\epsilon\geq 0$ number of partition levels using an octree-type structure. At each level of the partition, the spatial domain is divided into $2^n$ domains while the temporal domain is divided into $4$ subdomains so that \cref{eq_size_partition} is satisfied. The tree structure of the partition is defined as follows.
\begin{itemize}
\item At the level $L=0$, the domain $\smash{I_{1, \ldots, 1} \coloneqq \underbrace{I_1\times\ldots\times I_1}_{n\text{ times}}\times I_1=[0,1]^n\times [0,1]}$ is the root of the partitioning tree.
\item At a given level $0 \le L \le n_\epsilon-1$, if $I_{j_1,\ldots ,j_n,j_{n+1}}$ is a node of the tree, then it has $4\times 2^n$ children of the form
\[I_{2j_1+k_1,\ldots ,2j_n+k_n,4j_{n+1}+k_{n+1}},\quad k_1,\ldots k_n\in \{0,1\},\,k_{n+1}\in\{0,\ldots 3\}.\]
Here, if $I_{j_1} = [a,b]\subset [0,1]$, then $I_{2j_1} = [a,(a+b)/2]$ and $I_{2j_1+1} = [(a+b)/2, b]$. The division of the temporal interval $I_{j_{n+1}}$ into four subintervals is performed similarly.
\end{itemize}
The tree structure of the hierarchical partition in spatial dimension $n=1$ is displayed in \cref{fig_tree_partition} along with examples of admissible subdomains.

\begin{figure}[htbp]
\vspace{0.2cm}
\centering
\begin{overpic}[width=\linewidth]{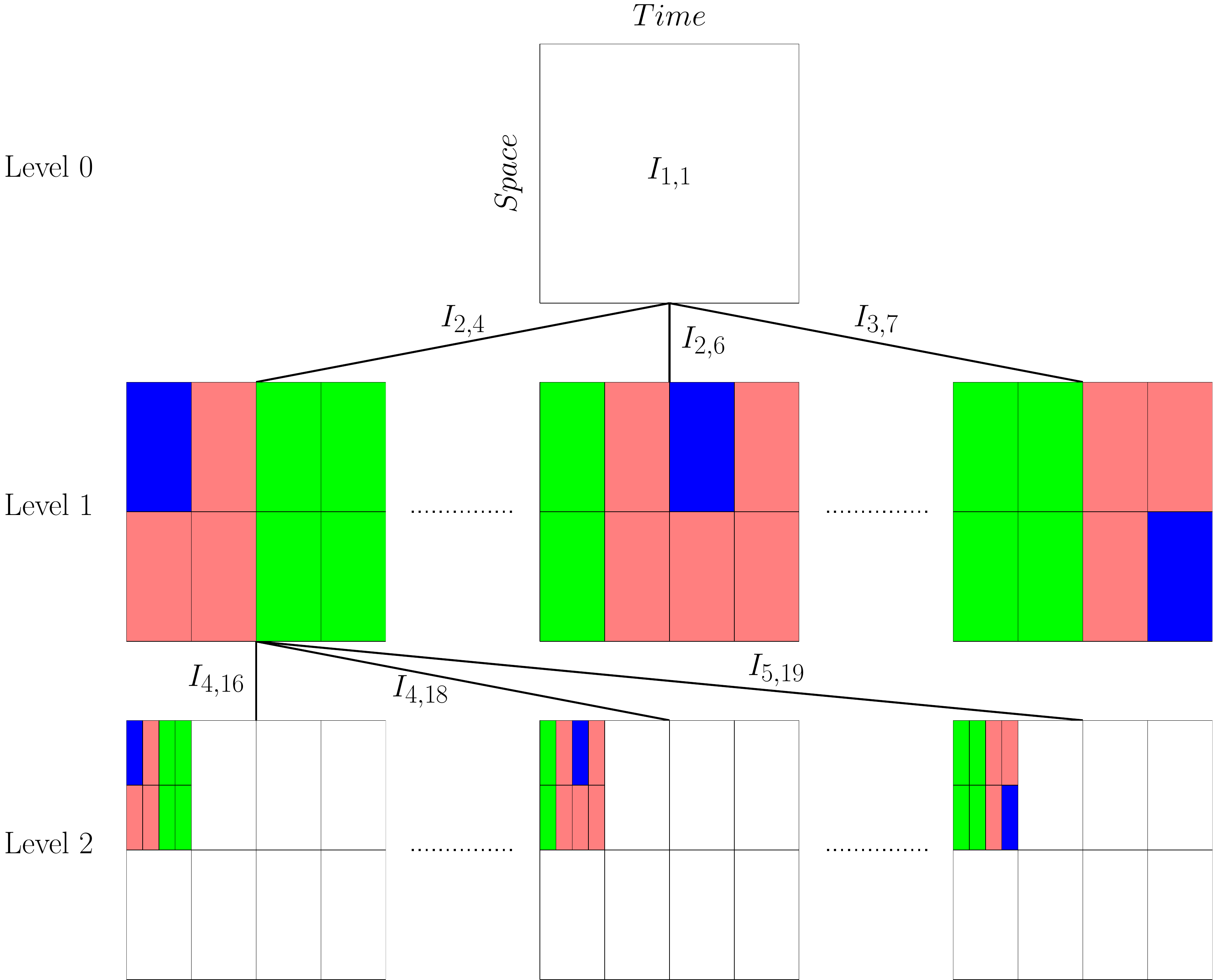}
\end{overpic}
\caption{Hierarchical partition of the domain $[0,1]\times [0,1]$, where the spatial domains are divided into $2$ at each level, and the temporal domains are divided into $4$. At the levels $1$ and $2$ of the tree, the domains coloured in red (resp.~green) are non-admissible (resp.~admissible) for the blue domain.}
\label{fig_tree_partition}
\end{figure}

Using the partition of $D\times I$, we can define a tree structure for $(D\times I)\times (D\times I)$ and cluster the tree nodes into admissible and non-admissible sets, respectively denoted by $P_{\textup{adm}}$ and $P_{\textup{non-adm}}$. These two sets also allow us to design a hierarchical partition of the domain $(D\times I)\times (D\times I)$.
\begin{itemize}
\item At the level $L=0$, the root of the tree is given by the domain $I_{1,\ldots 1}\times I_{1,\ldots 1}$, which belongs to the non-admissible set as it does not satisfy \cref{eq_adm_domain}.
\item At a given level $0 < L\leq n_\epsilon-1$, if $I_{j_1,\ldots ,j_{n+1}}\times I_{\tilde{j}_1,\ldots ,\tilde{j}_{n+1}}$ is a node of the tree, then it is either in the non-admissible set if all the respective indices are separated by at most one, i.e.,~$|j_1-\tilde{j}_1|\leq 1,\ldots,|j_{n+1}-\tilde{j}_{n+1}|\leq 1$, or labeled as admissible otherwise. If the node is admissible then it is added to the hierarchical partition. Otherwise, we subdivide it into $(4\times 2^n)^2$ children using cross-products of the $4\times 2^n$ children of $I_{j_1,\ldots ,j_{n+1}}$ and $I_{\tilde{j}_1,\ldots ,\tilde{j}_{n+1}}$ in the partition of $D\times [0,1]$.
\item At the final level $L=n_\epsilon$, we add both the admissible and non-admissible domains to the partition.
\end{itemize}

\begin{figure}[htbp]
\vspace{0.5cm}
\centering
\begin{overpic}[width=\linewidth]{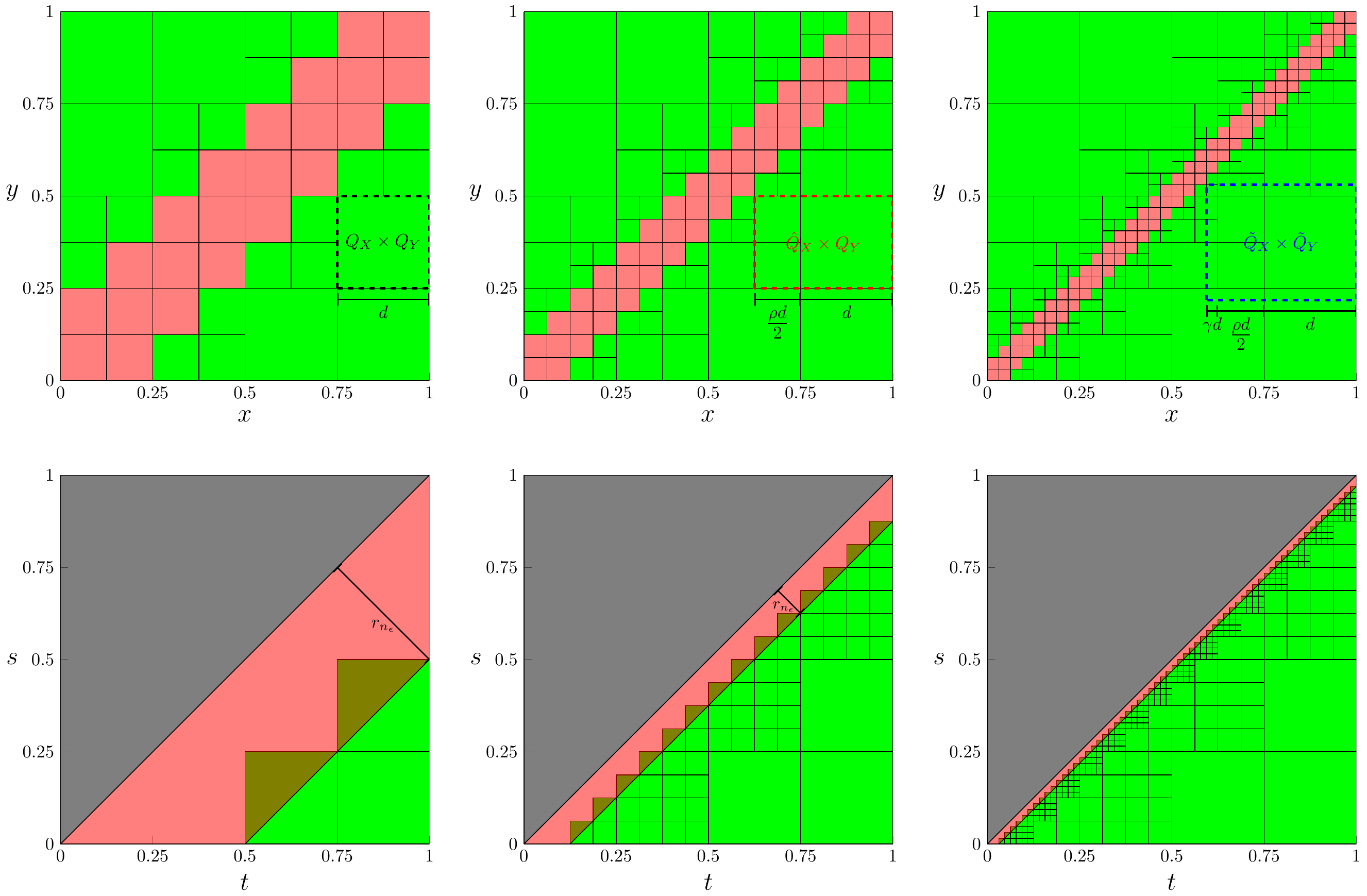}
\put(0,66){(a)}
\put(14,32){Level 2}
\put(48,32){Level 3}
\put(82,32){Level 4}
\put(14,66){Level 3}
\put(48,66){Level 4}
\put(82,66){Level 5}
\put(0,33){(b)}
\end{overpic}
\caption{(a) Illustration of level $3,4$, and $5$ of the hierarchical partition of the spatial domain $[0,1]\times [0,1]$. The green blocks are admissible domains with $\rho=1$ while the pink domains are non-admissible. The left to right panels display one admissible domain $Q_X\times Q_Y$ along with the elongated domains $\hat{Q}_X\times Q_Y$ (dashed red rectangle) and $\tilde{Q}_X\times Q_Y$ (dashed blue rectangle) appearing in the bounds of \cref{sec:green_entire}. (b) Illustration of the hierarchical partition of the temporal domain. At each level the non-admissible blocks are separated into $4^2$ domains. For partition level $n_\epsilon$, $r_{n_\epsilon}$ is the width of the non-admissible region.}
\label{fig_proof}
\end{figure}

In \cref{fig_proof}, we illustrate slices of the spatial and temporal partition of the domain $(D\times I)\times (D\times I)$ when $D = [0,1]$. The regions coloured in green are admissible while the ones coloured in red are non-admissible. In addition, the grey area in \cref{fig_proof}(b) shows the domain on which the Green's function is zero. At the final level $n_\epsilon$, all the non-admissible domains have the same diameter and, if $Q_X\times Q_Y$ is non-admissible with $Q_X = D_X\times I_X$, we have $\diam(Q_X) = \diam(D_X) = \diam(I_X)^{1/2} = 2^{-n_\epsilon}$. Therefore, the width of the non-admissible temporal region (see~\cref{fig_proof}(b)) is given by 
\begin{equation} \label{eq_size_domain}
r_{n_\epsilon}=\sqrt{2}\times 2^{-2n_\epsilon}.
\end{equation}

We now compute an upper bound on the number of admissible domains in the partition. By construction, the number of non-admissible regions in the hierarchical partition satisfies
\begin{equation} \label{eq_non_adm_dom}
|P_{\textup{non-adm}}| = (3\times 4^{n_\epsilon}-2)\times(3\times 2^{n_\epsilon}-2)^n.
\end{equation}
Moreover, the number of new admissible domains, $|P_{\textup{adm}}(L)|$, added to the partition at a given level $1\leq L\leq n_\epsilon$ is given by the number of children of the non-admissible domains at the previous level minus the number of non-admissible sets at the current level, i.e.,
\begin{equation} \label{eq_adm_dom_L}
|P_{\textup{adm}}(L)| = (4\times 2^n)^2 (3\times 4^{L-1}-2)(3\times 2^{L-1}-2)^n-(3\times 4^{L}-2)(3\times 2^{L}-2)^n,
\end{equation}
where the number of non-admissible sets is given by \cref{eq_non_adm_dom}. We can then compute the total number of admissible domains, $|P_{\textup{adm}}|$ in the partition by summing \cref{eq_adm_dom_L} over $1\leq L\leq n_\epsilon$ and obtain
\begin{align*}
|P_{\textup{adm}}| &= \sum_{L=1}^{n_\epsilon} 2^{2n+4} (3\times 4^{L-1}-2)(3\times 2^{L-1}-2)^n- (3\times 4^{L}-2)(3\times 2^{L}-2)^n\\
&= 1 - (3\times 4^{n_\epsilon}-2)(3\times 2^{n_\epsilon}-2)^n
+(2^{2n+4}-1)\sum_{L=0}^{n_\epsilon-1}  (3\times 4^{L}-2)(3\times 2^{L}-2)^n.
\end{align*}
We can bound $|P_{\textup{adm}}|$ by computing the sum of a geometric series as
\[\sum_{L=0}^{n_\epsilon-1}  (3\times 4^{L}-2)\times(3\times 2^{L}-2)^n\leq 3^{n+1}\sum_{L=0}^{n_\epsilon-1}(2^{n+2})^L=3^{n+1}\frac{2^{(n+2)n_\epsilon}-1}{2^{n+2}-1}\leq \left(\frac{3}{2}\right)^{n+1}\!\!2^{(n+2)n_\epsilon}.\]
The number of admissible domains is bounded by
\begin{equation} \label{eq_number_adm}
|P_{\textup{adm}}| \leq 2^{2n+4}3^{n+1}2^{-(n+1)}2^{(n+2)n_\epsilon}= 24\times 6^n 2^{(n+2)n_\epsilon}.
\end{equation}
We conclude the construction of the hierarchical partition of $\U\times \U$ by intersecting each element of the partition with $(\Omega\times I)\times (\Omega\times I)$. In the following section, we assume that we have already constructed a partition of $\U\times\U$ for a general bounded domain $\Omega\subset\R^n$ and $I=[0,T]$. In fact, the number of admissible domains and size of the non-admissible region (cf.~\cref{eq_number_adm,eq_size_domain}) remain the same up to a constant that depends on the size of the domain $\U$.

\subsection{Diagonal Estimate of Green's Functions} \label{sec_diagonal}

This section determines the number of hierarchical partitioning levels needed to neglect the Green's function on the non-admissible regions of the partition of $\U\times\U$ defined in \cref{sec_hier}. To start, we use a global Gaussian estimate for Green's functions associated with parabolic PDEs, which guarantees the existence of a positive constant $C>0$ such that the Green's function is positive and bounded by~\citep[Eq.~4.2]{cho2012global}
\begin{equation} \label{eq_gaussian_diag}
G(x,t,y,s)\leq C\frac{\Theta(t-s)}{(t-s)^{n/2}}\exp\left(-\frac{\kappa |x-y|^2}{t-s}\right),\quad (x,t)\neq(y,s)\in\U,
\end{equation}
where $\kappa>0$ is a constant depending on the uniform parabolicity constants~\eqref{eq_uni_parab} and is independent of $T$. We remark that the Green's function is zero if $t\leq s$, as can be seen in the gray regions of~\cref{fig_proof}(b). For a given diameter $0<r_t\leq T$, we define a diagonal subdomain of $\U\times \U$ as
\[\D_{r_t}\coloneqq \{(x,t,y,s)\in \U\times \U,\, |t-s|<r_t\}.\]
We use the estimate~\eqref{eq_gaussian_diag} to bound the Green's function on the domain $\D_{r_t}$ in the $L^p$-norm, for $p\geq 1$.

\begin{proposition} \label{lem_diag_bound}
Let $p\geq 1$, $0<r_t\leq T$, and assume that $n(p-1)<2$. Then, there exists a constant $C_{\textup{diag}}=C_{\textup{diag}}(\Omega,T,p,\kappa)>0$ such that
\[\|G\|_{L^p(\D_{r_t})}\leq C_{\textup{diag}} r_t^{[1-n(p-1)/2]/p}\|G\|_{L^p(\U\times \U)}.\]
\end{proposition}

\begin{proof}
Let $y\in\Omega$, $s\in (0,T)$, and denote $I_s = (s,\min(s+r_t,T))$. Integrating \cref{eq_gaussian_diag}, raised to the power $p$, on the domain $\Omega\times I_s \subset \U$ yields the following inequality,
\begin{align*}
\int_{I_s}\int_{\Omega} |G(x,t,y,s)|^p\d x\d t &\leq \int_{I_s}\int_{\Omega}\frac{(\Theta(t-s))^pC^p}{(t-s)^{pn/2}}\exp\left\{-p\frac{\kappa |x-y|^2}{t-s}\right\}\d x\d t\\
&\leq C^p\int_{0}^{r_t} \tilde{t}^{-pn/2} \int_{\R^n} e^{-p \kappa |x|^2/\tilde{t}}\d x\d\tilde{t},
\end{align*}
where we use the change of variables $\tilde{t} = t-s$. We then make the change of variables $\tilde{x} = x\sqrt{p\kappa/\tilde{t}}$ to obtain
\begin{align*}
\|G(\cdot, \cdot, y, s)\|_{L^p(\Omega\times I_s)}^p &\leq C^p\int_{0}^{r_t} \tilde{t}^{-pn/2} \left(\frac{\tilde{t}}{p}\right)^{n/2}\!\!\!\int_{\R^n} \!\! e^{- |\tilde{x}|^2}\d \tilde{x}\d \tilde{t} \leq C^p\left(\frac{\pi}{p}\right)^{n/2} \!\!\!\int_{0}^{r_t}\!\!\tilde{t}^{-n(p-1)/2}\d \tilde{t}\\
&\leq \left(\frac{\pi}{p}\right)^{n/2}\!\! \frac{C^p}{1-n(p-1)/2}r_t^{1-n(p-1)/2}.
\end{align*}
Integrating this expression over $y\in\Omega$ and $s\in (0,T)$ yields
\[\|G\|_{L^p(\D_{r_t})}^p\leq \left(\frac{\pi}{p}\right)^{n/2} \!\! \frac{|\Omega| T C^p}{1-n(p-1)/2}r_t^{1-n(p-1)/2}.\]
If $n(p-1)<2$, then the Green's function is in $L^p(\U\times \U)$, and we can introduce a constant $C_{\text{diag}}=C_{\text{diag}}(\Omega,T,p,\kappa)>0$, such that
\[\|G\|_{L^p(\D_{r_t})}\leq C_{\text{diag}} r_t^{[1-n(p-1)/2]/p} \|G\|_{L^p(\U\times \U)},\]
which concludes the proof.
\end{proof}

We conclude that the Green's function restricted to the domain $\D_{r_t}$ has a relative small norm if $r_t$ is small. Applying \cref{lem_diag_bound} with the parameter $p=1$ gives the following $L^1$ estimate:
\begin{equation} \label{eq:green_1D_norm_bound}
\|G\|_{L^1(\D_{r_t})}\leq C_{\text{diag}} r_t\|G\|_{L^1(\U\times \U)}.
\end{equation}
Due to the hierarchical partition of the temporal domain, we can easily bound the norm of $G$ on non-admissible sets by using a temporal radius of $r_t = r_{n_\epsilon}=\sqrt{2}\times 2^{-2n_\epsilon}$ up to a constant depending on the size of the domain $\U$, where $n_\epsilon$ is the number of hierarchical levels (see.~\cref{eq_size_domain}). Then, since $P_{\text{non-adm}}\subset \D_{r_t}$ as illustrated by \cref{fig_proof}, \cref{eq:green_1D_norm_bound} yields
\begin{equation} \label{eq_bound_G_non_adm}
\|G\|_{L^1(P_{\text{non-adm}})} \le \sqrt{2}C_{\text{diag}} 2^{-2n_\epsilon}\|G\|_{L^1(\U \times \U)}\leq \epsilon \|G\|_{L^1(\U\times\U)},
\end{equation}
where we choose $n_\epsilon\sim (1/2)\log_2(1/\epsilon)$ hierarchical levels such that $\sqrt{2}C_{\text{diag}} 2^{-2n_\epsilon}\leq \epsilon$. This means that we can safely approximate $G$ with the zero approximant on non-admissible domains, and still get an approximation of $G$ within a relative accuracy of $\epsilon$ in the $L^1$ norm.

One might be able to slightly improve \cref{lem_diag_bound} by computing the integral of each non-admissible domain, similarly to the elliptic case~\citep{boulle2021learning}. However, the gain is expected to be marginal since the decay of the bound in \cref{eq_gaussian_diag} is essentially controlled by a well-separation of the temporal variables $t$ and $s$.

\subsection{Approximating Green's Functions on Admissible Domains} \label{sec:green_1D_adm}

The approximation of Green's functions on the well-separated domains of the partition of $\U\times \U$ is achieved using the randomized SVD for HS operators described in \cref{sec:rand_svd}. Let $Q_X\times Q_Y\in P_\text{adm}$ be an admissible domain and $k=k_\epsilon = c_{\rho/2}^{3} \lceil{\log\frac{1}{\epsilon}}\rceil^{n+3} + \lceil{\log \frac{1}{\epsilon}}\rceil$ be a target rank derived in~\cref{thm01}, we can combine \cref{thm01} and the Eckart--Young--Mirsky theorem~\citep[Thm.~4.4.7]{hsing2015theoretical}, which characterizes the best rank-$k$ approximation error to a HS operator, to bound the singular values of the Green's function restricted to $Q_X\times Q_Y$ by
\begin{equation} \label{eq:green_1D_sv} 
\left( \sum_{j=k_\epsilon+1}^\infty \sigma_{j,Q_X\times Q_Y}^2 \right)^{1/2} \le \|G-G_{k_\epsilon}\|_{L^2(Q_X \times Q_Y)} \le \epsilon \|G\|_{L^2(\hat{Q}_X \times Q_Y)}, 
\end{equation}
where $\sigma_{j,Q_X\times Q_Y}$ are the singular values of $G$ restricted to $Q_X\times Q_Y$. We conclude that the singular values of $G$ decay rapidly to $0$ on admissible domains.

With the rapidly decaying singular values, we can follow the arguments in~\citep[Sec.~4.1.2]{boulle2021learning} to use the randomized SVD for learning Green's functions on admissible sets. Roughly speaking, we start with a Gaussian process on $\U$ and define a covariance kernel $K$ that restricts onto $Q_Y \times Q_Y$. We then extend the restricted operator by $0$ on $\U\times \U$ and apply the randomized SVD. As a result, with a target rank of $k_\epsilon=c_{\rho/2}^{3} \lceil{\log\frac{1}{\epsilon}}\rceil^{n+3} + \lceil{\log \frac{1}{\epsilon}}\rceil$, an oversampling parameter $p = k_\epsilon$, and $t=e$, we combine \cref{thm01,eq:green_1D_sv} to obtain an approximant $\tilde{G}_{X\times Y}$ of $G$ on $Q_X\times Q_Y$ such that
\begin{equation} \label{eq:green_norm_adm_exp} 
\|G-\tilde{G}_{X\times Y}\|_{L^2(Q_X \times Q_Y)} \leq \left(1+se\sqrt{\frac{6k_\epsilon}{\gamma_{k_{\epsilon},Q_X\times Q_Y}}\frac{\Tr(K)}{\lambda_1}}\,\right)\epsilon \|G\|_{L^2(\hat{Q}_X \times Q_Y)},
\end{equation}
which holds with probability greater than $1-e^{-k_{\epsilon}}-e^{-k_\epsilon(s^2-2\log(s)-1)}\geq 1-2e^{-k_{\epsilon}}$ when $s\geq 3$. The factor $\gamma_{k_{\epsilon},Q_X\times Q_Y}$ characterizes the suitability of the covariance kernel for learning $G$ on the domain $Q_X\times Q_Y$. In this way, our algorithm requires $N_{\epsilon, X\times Y} = 2(k_\epsilon+p) =  \mathcal{O}\left(\log^{n+3}(1/\epsilon)\right)$ input-output pairs to learn an approximant to $G$ on $Q_X\times Q_Y$.

\begin{remark}
To apply the projection operator associated with the randomized approximation on the left of the HS operator in \cref{eq:MainProbabilityBound}, we need to solve the adjoint equation associated with \cref{eq:parabolic}, which is allowed by~\cref{ass_adjoint}.
\end{remark}

\subsection{Recovering the Green's Function on the Entire Domain} \label{sec:green_entire}

We can now recover the Green's function $G$ on the entire domain $\U \times \U$ and compute the number of input-output pairs needed to approximate it within accuracy $\epsilon>0$. With $n_\epsilon$ computed in \cref{sec_diagonal}, we can follow the arguments in~\citep[Sec.~4.4.1]{boulle2021learning} to quantify the total number of input-output pairs we need to approximate $G$ using the randomized SVD described in~\cref{sec:green_1D_adm}. In particular, we denote the worst $\gamma_{k_\epsilon,Q_X \times Q_Y}$ by
\begin{equation} \label{eq_gamma}
\Gamma_\epsilon = \min \{ \gamma_{k_\epsilon,Q_X \times Q_Y}, Q_X \times Q_Y \in P_{\rm{adm}} \},
\end{equation}
so that we need
\[ N_\epsilon = \mathcal{O}(|P_{\rm{adm}}| \log^{n+3}(1/\epsilon)) = \mathcal{O}(\epsilon^{-\frac{n+2}{2}}\log^{n+3}(1/\epsilon))\]
input-output pairs to capture the Green's function on admissible domains with $n_\epsilon\sim(1/2)\log_2(1/\epsilon)$ hierarchical levels (see~\cref{sec_diagonal}), and the number of admissible domains given by~\cref{eq_number_adm}.

We now provide an explicit bound for the approximation $\tilde{G}$ if we use zero approximant on non-admissible sets and learn with $N_\epsilon = \mathcal{O}(\epsilon^{-\frac{n+2}{2}}\log^{n+3}(1/\epsilon))$ many input-output pairs on admissible domains. First, we separate the norm error into error on admissible sets and that on non-admissible sets as
\begin{equation} \label{eq_global_norm}
\begin{aligned}
  \|G-\tilde{G}\|_{L^1(\U \times \U)} &\leq \|G-\tilde{G}\|_{L^1(P_{\rm{non-adm}})} + \|G-\tilde{G}\|_{L^1(P_{\rm{adm}})}\\
&\leq \epsilon \|G-\tilde{G}\|_{L^1(\U\times \U)} + \sum_{Q_X \times Q_Y \in P_{\rm{adm}}} \|G-\tilde{G}\|_{L^1(Q_X \times Q_Y)},
\end{aligned}
\end{equation}
where the second inequality comes from \cref{eq_bound_G_non_adm}.
Let $Q_X \times Q_Y\in P_{\rm{adm}}$, we focus on bounding the error on this subdomain with \cref{eq:green_norm_adm_exp}. Using H\"older's inequality, we have 
\[\norm{G(\cdot,\cdot,y,s)-G_k(\cdot,\cdot,y,s)}_{L^1(Q_X)} \le \abs{Q_X}^{1/2}\norm{G(\cdot,\cdot,y,s)-G_k(\cdot,\cdot,y,s)}_{L^2(Q_X)}, \quad (y,s)\in Q_Y,\]
which implies that
\begin{equation} \label{eq_L2_holder}
\|G-\tilde{G}\|_{L^1(Q_X \times Q_Y)}\leq \abs{Q_X}^{1/2}\abs{Q_Y}^{1/2}\|G-\tilde{G}\|_{L^2(Q_X \times Q_Y)}.
\end{equation}
We then apply \cref{eq:green_norm_adm_exp} to estimate the term $\|G-\tilde{G}\|_{L^2(Q_X \times Q_Y)}$ in \cref{eq_L2_holder} and bound the resulting right-hand side term $\|G\|_{L^2(\hat{Q}_X \times Q_Y)}$ by an $L^1$-estimate to complete the bound of the approximation error on admissible domains.
Let $\gamma>0$ be an arbitrary constant and define $\tilde{Q}_X\coloneqq \{X\in \U,\, \dist(X,\hat{Q}_X)\leq \frac{\gamma}{2}\diam \hat{Q}_X\}$.
We first remark that for $(y,s) \in Q_Y$, $G(\cdot,\cdot,y,s)$ satisfies $\mathcal{P} G(\cdot,\cdot,y,s)=0$ in $(\tilde D \cap \Omega)\times \tilde I$ and vanishes on $(\tilde D \cap \partial\Omega)\times \tilde I$, where $\tilde{Q}_X = (\tilde{D}\cap \Omega)\times\tilde{I}$. Therefore, by Moser's local maximum estimate~\citep[Thm.~6.30]{lieberman1996second}, we have
\begin{equation} \label{eq_moser_estimate}
\norm{G(\cdot,\cdot, y,s)}_{L^\infty(\hat Q_X)} \le C_1\abs{\tilde{Q}_X}^{-1} \norm{G(\cdot,\cdot,y,s)}_{L^1(\tilde{Q}_X)},
\end{equation}
where $C_1=C_1(n, \lambda, \Lambda, \gamma)>0$. \cref{eq_moser_estimate} implies that for all $(y,s) \in Q_Y$, we have
\begin{equation} \label{eq12.55sat}
\int_{\hat Q_X}\abs{G(x,t,y,s)}^2 \d x \d t \le C_1^2 \abs{\tilde{Q}_X}^{-2} \abs{\hat Q_X} \left(\int_{\tilde{Q}_X} \abs{G(x,t,y,s)} \d x \d t\right)^2.    
\end{equation}
By integrating over $(y,s) \in Q_Y$ and using an integral version of Minkowski's inequality, we obtain
\begin{multline}
\int_{Q_Y}\int_{\hat Q_X}\abs{G(x,t,y,s)}^2\d x \d t \d y \d s \le C_1^2 \abs{\tilde{Q}_X}^{-2} \abs{\hat Q_X} \int_{Q_Y}\left(\int_{\tilde{Q}_X} \abs{G(x,t,y,s)}\d x\d t\right)^2 \d y\d s \\
\le C_1^2 \abs{\tilde{Q}_X}^{-2} \abs{\hat Q_X} \left\{ \int_{\tilde{Q}_X}\left(\int_{Q_Y} \abs{G(x,t,y,s)}^2 \d y \d s\right)^{1/2} \d x \d t \right\}^2. \label{eq12.56sat}
\end{multline}
On the other hand, for $(x,t) \in \tilde{Q}_X$, $G(x,t,\cdot,\cdot)$ satisfies $\mathcal{P}^* G(x,t,\cdot,\cdot) =0$ in $Q_Y$, where $\mathcal{P}^*$ is the adjoint operator of $\mathcal{P}$~\citep{cho2008green}. Similarly to \cref{eq_moser_estimate,eq12.55sat}, we have
\begin{equation} \label{eq13.06sat}
\int_{Q_Y}\abs{G(x,t,y,s)}^2\d y \d s \le C_2^2 \abs{\tilde{Q}_Y}^{-2} \abs{Q_Y} \left(\int_{\tilde{Q}_Y} \abs{G(x,t,y,s)}\d y \d s\right)^2, 
\end{equation}
where $\tilde{Q}_Y\coloneqq \{Y\in \U,\, \dist(Y,Q_Y)\leq \frac{\gamma}{2}\diam Q_Y\}$ and $C_2>0$ is a constant. Combining \eqref{eq12.56sat} and \eqref{eq13.06sat} yields
\begin{equation*}
\int_{Q_Y}\int_{\hat Q_X}\!\! \abs{G(x,t,y,s)}^2\d x\d t \d y\d s \le C_1^2 C_2^2 \frac{ \abs{\hat Q_X} \abs{Q_Y}}{\abs{\tilde{Q}_X}^2\abs{\tilde{Q}_Y}^{2}} 
\left\{\int_{\tilde{Q}_X} \int_{\tilde{Q}_Y}\!\! \abs{G(x,t,y,s)}\d y \d s \d x \d t \right\}^2,
\end{equation*}
or equivalently,
\begin{equation} \label{eq13.17sat}
\norm{G}_{L^2(\hat Q_X\times Q_Y)}  \le C_1 C_2 \abs{\tilde{Q}_X}^{-1} \abs{\hat Q_X}^{1/2} \abs{\tilde{Q}_Y}^{-1} \abs{Q_Y}^{1/2}
\norm{G}_{L^1(\tilde{Q}_X\times \tilde{Q}_Y)}.
\end{equation}
Finally, the multiplication of \cref{eq13.17sat} by the term $\abs{Q_X}^{1/2}\abs{Q_Y}^{1/2}$ from \cref{eq_L2_holder} yields
\[\abs{Q_X}^{1/2}\abs{Q_Y}^{1/2} \norm{G}_{L^2(\hat Q_X\times Q_Y)}\leq C \norm{G}_{L^1(\tilde{Q}_X\times \tilde{Q}_Y)},\]
where $C>0$ is a constant, because $Q_X\subset \hat{Q}_X\subset\tilde{Q}_X$ and $Q_Y\subset\tilde{Q}_Y$. 

We then choose $\gamma=1/16$ so that the domain $\tilde{Q}_X\times \tilde{Q}_Y$ is included in a finite number, $C_n$, of neighbors in the hierarchical partition of the domain including itself (see \cref{fig_proof}(a)). Combining~\cref{eq_global_norm,eq:green_norm_adm_exp} and the $L^1$ argument described in the previous paragraph, we conclude that
\begin{equation} \label{eq_final_L1}
\begin{aligned}
\|G-\tilde{G}\|_{L^1(\U \times \U)} &\leq \epsilon \|G\|_{L^1(\U \times \U)} + \sum_{Q_X \times Q_Y \in P_{\rm{adm}}} \|G-\tilde{G}\|_{L^1(Q_X \times Q_Y)} \\
&\leq \epsilon \|G\|_{L^1(\U \times \U)} + \sum_{Q_X \times Q_Y \in P_{\rm{adm}}}  s C_{\text{svd}}  k_\epsilon^{1/2}\Gamma_{\epsilon}^{-1/2}\epsilon\|G\|_{L^1(\tilde{Q}_X \times \tilde{Q}_Y)}\\
&\le s (2C_n+1) C_{\text{svd}}  k_\epsilon^{1/2}\Gamma_{\epsilon}^{-1/2}\epsilon\|G\|_{L^1(\U \times \U)},
\end{aligned}
\end{equation}
where $C_{\text{svd}}>0$ is a constant.

Finally, we choose $s=3$ to obtain a probability of failure of the randomized SVD less than $2e^{-k_\epsilon}$ on each admissible domain (cf.~\cref{sec:green_1D_adm}). Following \cref{eq_final_L1}, as $\epsilon \rightarrow 0$, the global approximation error between $G$ and the constructed approximant $\tilde{G}$ on $\U \times \U$ satisfies
\[ \|G-\tilde{G}\|_{L^1(\U \times \U)} = \mathcal{O}(\Gamma_\epsilon^{-1/2}\log^{(n+3)/2}(1/\epsilon)\epsilon)\,\|G\|_{L^1(\U \times \U)}, \]
with probability $\geq (1-2e^{-k_\epsilon})^{24^{n+1}\times 2^{(n+2)n_\epsilon}}=
1-\mathcal{O}(\epsilon^{\log^{n+2}(1/\epsilon)-\frac{n+2}{2}})$. This indicates that $\tilde{G}$ is a good approximation of $G$ with high probability. We then make the change of variable $\tilde{\epsilon}\coloneqq \epsilon\log^{(n+3)/2}(1/\epsilon)$ to obtain the bound
\[ \|G-\tilde{G}\|_{L^1(\U \times \U)} = \mathcal{O}(\Gamma_{\tilde{\epsilon}}^{-1/2}\tilde{\epsilon})\,\|G\|_{L^1(\U \times \U)},\]
with probability $\geq 1-\mathcal{O}(\tilde{\epsilon}^{\log^{n+1}(1/\tilde{\epsilon})})$ for $\epsilon$ small enough. Note that the factor $\Gamma_{\tilde{\epsilon}}$ is changed according to its implicit definition given by \cref{eq_gamma}. As a result, the number of input-output pairs is given by
\[N_{\tilde{\epsilon}}=\mathcal{O}(\tilde{\epsilon}^{-\frac{n+2}{2}}\log^{(n+3)(1-\frac{n+2}{4})}(1/\tilde{\epsilon}))=\mathcal{O}(\tilde{\epsilon}^{-\frac{n+2}{2}}\log(1/\tilde{\epsilon})),\]
and we can drop the tilde symbol to conclude the proof of \cref{thm_learning_rate}.

\section{Summary and Discussion} \label{sec_conclusions}

We derive a rigorous learning rate for parabolic operators, by giving an upper bound on the number of training data needed to learn Green's functions within a prescribed relative accuracy. Our analysis relies on an extension of a result from~\cite{bebendorf2003existence} to show that Green's functions of parabolic operators admit low-rank properties on well-separated domains, i.e., away from the singularity near the diagonal. A similar low rank property is derived for elliptic operators in dimension three. This result may motivate the development of novel algorithms that use the hierarchical structures of parabolic operators to discretize time-dependent equations. One interesting outcome of this work is that the analysis and the resulting approximation error bounds are obtained using the $L^1$-norm since Green's functions of parabolic operators in spatial dimension greater than one are usually not square-integrable. This fact may partially explain the challenges met by the current deep learning techniques that attempt to learn the solution operators of time-dependent mathematical models using a quadratic loss function. Hence, \cite{Krishnapriyan2021mode} and~\cite{wang2022and} recently observed and analysed mode failure issues in existing physics-informed neural network architectures. The development of PDE learning techniques based on the $L^1$ loss and NN architectures exploiting singularities of the underlying model, such as rational NNs~\citep{boulle2020rational}, is of significant interest for the field to overcome the challenges resulting from learning PDEs with short-lived transient dynamics. Finally, we note that the analysis performed in this paper is applicable to obtain a learning rate for elliptic PDEs in any spatial dimension in $L^1$-norm. This generalizes the previous results from~\cite{boulle2021learning}, and concludes the study of elliptic and parabolic PDEs. However, Green's functions associated with hyperbolic PDEs do not admit a similar low-rank structure on well-separated domains due to singularity near characteristics lines. Thus, the theoretical analysis remains a future challenge.

\acks{N.B. was supported by the EPSRC Centre for Doctoral Training in Industrially Focused Mathematical Modelling through grant EP/L015803/1 in collaboration with Simula Research Laboratory. S.K. was supported by National Research Foundation of Korea grants NRF-2019R1A2C2002724 and NRF-2022R1A2C1003322. T.S. and A.T. were supported by the National Science Foundation grants DMS-1818757, DMS-1952757, and DMS-2045646.}




\vskip 0.2in
\bibliography{biblio}

\end{document}